\documentclass[11pt,letter]{amsart}
\usepackage[all,cmtip]{xy}
\usepackage{enumerate, comment}
\usepackage{mathrsfs}
\usepackage{ amssymb, latexsym, amsmath}
\usepackage{fullpage}
\usepackage{url}
\usepackage{hyperref}
\hypersetup{
 colorlinks=true,
 linkcolor=blue,
 filecolor=magenta,
 citecolor=brown,
 urlcolor=cyan,
 pdftitle={Kida's formula},
 }
\usepackage{helvet}
\usepackage{amsfonts}
\usepackage{tikz}
\usepackage{commath}
\usepackage{enumitem}
\usepackage{amsthm}
\usepackage{amscd}
\usepackage{stmaryrd}
\usepackage[mathscr]{eucal}
\usepackage{indentfirst}
\usepackage{graphicx}
\usepackage{tikz-cd}
\usepackage{graphics}
\usepackage{pict2e}
\usepackage{epic}
\usepackage{mathtools}

\newcommand{\closure}[2][3]{%
{}\mkern#1mu\overline{\mkern-#1mu#2}}
\usepackage[normalem]{ulem}
\usepackage[OT2,T1]{fontenc}
\numberwithin{equation}{section}
\usepackage[margin=3.4cm]{geometry}
 
\usepackage{color}
\DeclareSymbolFont{cyrletters}{OT2}{wncyr}{m}{n}
\DeclareMathSymbol{\Sha}{\mathalpha}{cyrletters}{"58}

\newcommand{\cF}{\mathcal{F}}
\newcommand{\cL}{\mathcal{L}}
\newcommand{\cH}{\mathcal{H}}
\newcommand{\cP}{\mathcal{P}}
\newcommand{\cA}{\mathcal{A}}

\newcommand{\cyc}{\op{cyc}}
\newcommand{\ac}{\op{ac}}
\newcommand{\Hom}{\op{Hom}}

\newcommand{\ord}{\op{ord}}
\newcommand{\corank}{\op{corank}}

\newcommand{\GL}{\op{GL}_2}
\newcommand{\Gal}{\op{Gal}}
\newcommand{\Sel}{\op{Sel}}

\newcommand{\Z}{\mathbb{Z}}
\newcommand{\Q}{\mathbb{Q}}
\newcommand{\Kinf}{K_{\infty}}
\newcommand{\Linf}{L_{\infty}}
\newcommand{\F}{\mathbb{F}}

\newcommand{\mupK}{\mu(E/\Kinf)}
\newcommand{\mupL}{\mu(E/\Linf)}
\newcommand{\lambdapK}{\lambda(E/\Kinf)}
\newcommand{\lambdapL}{\lambda(E/\Linf)}
\newcommand{\Kan}{K^{\ac}}
\newcommand{\SelpL}{\Sel(E/L_{\infty})}
\newcommand{\SelpK}{\Sel(E/K_{\infty})}
\newcommand{\op}[1]{\operatorname{#1}}

\theoremstyle{plain}
 \theoremstyle{definition}
\newtheorem{Th}{Theorem}[section]
\newtheorem{Lemma}[Th]{Lemma}

\newtheorem {lthm}{Theorem} 
\newtheorem{Cor}[Th]{Corollary}
\newtheorem{Proposition}[Th]{Proposition}
\newtheorem{assumption}[Th]{Assumption}
\newtheorem{Remark}[Th]{Remark}
\theoremstyle{definition}
\newtheorem{Defi}[Th]{Definition}

\begin{document}

\title{Iwasawa Invariants for elliptic curves over $\Z_{p}$-extensions and Kida's Formula}
\author{Debanjana Kundu}
\address{Department of Mathematics \\ University of British Columbia \\
 Vancouver BC, V6T 1Z2, Canada.} 
 \email{dkundu@math.ubc.ca}
\author{Anwesh Ray}
\address{Department of Mathematics \\ University of British Columbia \\
 Vancouver BC, V6T 1Z2, Canada.} 
 \email{anweshray@math.ubc.ca}


\begin{abstract}
This paper aims at studying the Iwasawa $\lambda$-invariant of the $p$-primary Selmer group.
We study the growth behaviour of $p$-primary Selmer groups in $p$-power degree extensions over non-cyclotomic $\Z_p$-extensions of a number field.
We prove a generalization of Kida's formula in such a case.
Unlike the cyclotomic $\Z_p$-extension, where all primes are finitely decomposed, in the $\Z_p$-extensions we consider, primes may be infinitely decomposed.
In the second part of the paper, we study the relationship of Iwasawa invariants with respect to congruences, obtaining refinements of the results of R.~Greenberg--V.~Vatsal and K.~Kidwell.
As an application, we provide an algorithm for constructing elliptic curves with large anticyclotomic $\lambda$-invariant.
Our results are illustrated by explicit computation.
\end{abstract}

\maketitle
\section{Introduction}
The classical Riemann--Hurwitz formula describes the relationship of the Euler characteristics of two surfaces when one is a ramified covering of the other.
Suppose $\pi: R_1 \rightarrow R_2$ is an $n$-fold covering of compact, connected Riemann surfaces and $g_1, \ g_2$ are their respective genus.
The classical Riemann--Hurwitz formula is the statement
\[
2g_2 -2 = (2g_1-2)n + \sum\left(e(P_2)-1\right)
\] 
where the sum is over all points $P_2$ on $R_2$ and $e(P_2)$ denotes the ramification index of $P_2$ for the covering $\pi$ (see \cite[Chapter II Theorem~5.9]{Sil09}).
An analogue of the above formula for algebraic number fields was proven by Y.~Kida in \cite{Kid80}.
Kida's formula describes the change of (cyclotomic) Iwasawa $\lambda$-invariants in a $p$-extension in terms of the degree and the ramification index.
In \cite{Iwa81}, K.~Iwasawa proved this formula using the theory of Galois cohomology for extensions of $\Q$ which are not necessarily finite.
More precisely,

\begin{lthm}[{\cite[Theorem~6]{Iwa81}}]
\label{thm: classical Kida's formula}
Let $p\geq 2$ and $K$ be a number field.
Let $K^{\cyc}$ be the cyclotomic $\Z_p$-extension of $K$ and $\cL/ K^{\cyc}$ be a cyclic extension of degree $p$, unramified at every infinite place of $K_{\cyc}$.
Assume that the classical $\mu$-invariant, $\mu(K^{\cyc})=0$.
Then
\[
\lambda(\cL) = p\lambda(K^{\cyc}) + \sum_{w^{\prime}}\left(e(w^{\prime}| v^{\prime})-1\right) + (p-1)(h_2 -h_1).
\] 
The sum is over all primes $w^{\prime}$ of $\cL$ (above $v^\prime$ in $K^{\cyc}$) not above $p$, $h_i$ is the rank of the abelian group $H^i(\cL/K_{\cyc}, U(\cL))$, and $U(\cL)$ is the group of all units of $\cL$.
\end{lthm}

In the study of rational points of elliptic curves, the Selmer group plays a crucial role.
In \cite{Maz72}, B.~Mazur initiated the study of the growth of the $p$-primary Selmer group in $\Z_p$-extension of number fields.
In \cite{HM99}, Y.~Hachimori and K.~Matsuno proved an analogue of Theorem~\ref{thm: classical Kida's formula} for $p$-primary Selmer groups of elliptic curves in $p$-power extensions of a (fixed) number field over the \emph{cyclotomic} $\Z_p$-extension.
This result has been generalized in several directions.
For modular forms over the \emph{cyclotomic} $\Z_p$-extension, analogous results have been worked out for signed Selmer groups at non-ordinary primes by J.~Hatley--A.~Lei, see \cite[Theorem~6.7]{HL19}.
It has been extended to a general class of Galois representations over the \emph{cyclotomic} $\Z_p$-extension, including the case of $p$-ordinary Hilbert modular forms and $p$-supersingular modular forms by R.~Pollack--T.~Weston in \cite{PW06}.
The first named author studied Kida-type formula for fine Selmer groups of elliptic curves over the \emph{cyclotomic} $\Z_p$-extension of totally real number fields in \cite{Kun21}.
Kida-type formulae for $p$-primary Selmer groups of elliptic curves have also been proven over special classes of non-abelian $p$-adic Lie extensions containing the \emph{cyclotomic} $\Z_p$-extension by A.~Bhave in \cite{Bha07}.
In the first part of this article, we prove a Kida-like formula for $p$-primary Selmer groups of elliptic curves in \emph{more general} $\Z_p$-extensions (see Theorem~\ref{main theorem}).

For ease of exposition, in the introduction, we state our main result in a simplified setting.
Let $p\geq 5$ be a fixed rational prime.
Let $E_{/\Q}$ be an elliptic curve without complex multiplication and with good ordinary reduction at $p$.
Let $K$ be a fixed imaginary quadratic field and $\Kan$ denote the anticyclotomic $\Z_p$-extension of $K$.
Let $L/K$ be a Galois extension of $p$-power degree disjoint from $\Kan$ and $L_\infty$ be the compositum $L\cdot \Kan$.
Suppose that the $p$-primary Selmer group over $\Kan$ (and $\Linf$) is cofinitely generated as a $\Z_p$-module.
Then our result relates their respective $\lambda$-invariants.
\begin{lthm}\label{theorem1}
With the setting as above,
\[\lambda(E/\Linf)=[L_{\infty}:\Kan] \lambda(E/\Kan)+\sum_{w'\in \cP_1} \left(e(w^\prime|v^\prime)-1\right)+\sum_{w'\in \cP_2} 2\left(e(w^\prime|v^\prime)-1\right)\]
where $e(w^\prime|v^\prime)$ is the ramification index of $w^\prime$ above a prime $v^\prime$ (in $\Kan$) and the sets $\cP_1, \ \cP_2$ are as in Definition \ref{sets Pi}.
\end{lthm}
We remark that the $p$-primary Selmer group need \emph{not always} be $\Lambda$-cotorsion over the anticyclotomic $\Z_p$-extension, $\Kan$ (see \cite{Ber95, Cor02, Vat03}).
However, by the work of Vatsal \cite{Vat02}, M.~Bertolini--H.~Darmon \cite{BD05}, and Pollack--Weston \cite[Theorem~5.3]{PW11}, it is known that if some natural conditions are satisfied, then the $p$-primary Selmer group is not only $\Lambda$-cotorsion, but also cofinitely generated as a $\Z_p$-module.
We show (in Corollary~\ref{cor: mu=0 for L}) that this property propagates to all $p$-power extensions $L$ of $K$ (which are disjoint from $\Kan$) satisfying the additional condition that if $E$ has split multiplicative reduction at $\ell$ and $v|\ell$ is inert in $K$, then $v$ splits completely in $L$.
We expect that the aforementioned result can be extended to the case of plus/minus Selmer groups when $p$ is a prime of supersingular reduction under reasonable hypotheses introduced by B.~D.~Kim in \cite{Kim13}.

Two elliptic curves $E_1$ and $E_2$ over a number field $K$ are said to be $p$-congruent if their associated residual representations are isomorphic, i.e., $E_1[p]$ and $E_2[p]$ are isomorphic as $\op{Gal}(\closure{K}/K)$-modules.
There is much interest in investigating the behaviour of Iwasawa invariants for congruent Galois representations.
Such investigations were initiated by Greenberg--Vatsal in \cite{GV00}.
They considered the case when both $E_1$ and $E_2$ are defined over $\Q$, and showed how the (cyclotomic) Iwasawa $\mu$ and $\lambda$-invariants for the $p$-primary Selmer groups of $E_1$ and $E_2$ are related.
Let $\mathfrak{T}$ be the set of primes $v\nmid p$ of $K$ at which $E_1$ or $E_2$ has bad reduction.
Over $K^{\cyc}$, Greenberg--Vatsal compare the $\mathfrak{T}$-imprimitive Selmer groups of $E_1$ and $E_2$.
These results were conditionally generalized by Kidwell for $\Z_p$-extensions of number fields (see \cite[Theorem~6.1]{Kid18}).
This requires a finiteness assumption, which is satisfied if primes in $\mathfrak{T}$ are finitely decomposed in the $\Z_p$-extension, $\Kinf$.
We refine the results of Greenberg--Vatsal and Kidwell in two ways.
We show that the set of primes $\mathfrak{T}$ can be replaced by a smaller set $\Omega_0$ (see Definition \ref{def of Omega0}).
Having a smaller set of primes $\Omega_0$ to consider makes for a more refined relationship between $\lambda$-invariants for the Selmer groups of $E_1$ and $E_2$.
We show that
\begin{equation}\label{lambda minus lambda}
\lambda(E_1/\Kinf) -\lambda(E_2/\Kinf)= \sum_{v\in \Omega_0}  \left(\sigma_{E_2}^{(v)} - \sigma_{E_1}^{(v)}\right),
\end{equation}
where $\sigma_{E_i}^{(v)}$ is the corank of a certain locally defined $\Lambda$-module (see Definition \ref{random defi needed to be cited}).
We also clarify the assumptions on the splitting of primes in $\Z_p$-extensions.
These assumptions apply only to the smaller set of primes $\Omega_0$ and not to the set $\mathfrak{T}$.
The approach used here is similar to that of the second named author and R.~Sujatha, who prove such results for the cyclotomic $\Z_p$-extension in \cite{ray20}.

As an application of our results on congruences, we provide an algorithm for constructing elliptic curves whose Selmer group has a large $\lambda$-invariant.
Such elliptic curves are considered over the anticyclotomic $\Z_p$-extension of an imaginary quadratic field.
In \cite{Mat07}, this question has been studied for some small primes (i.e., $p\leq 7$ or $p=13$) for elliptic curves with good \emph{ordinary} reduction at $p$, over the cyclotomic $\Z_p$-extension (of $\Q$).
In \cite{Kim09}, elliptic curves over $\Q$ with \emph{supersingular} reduction at $p=3$ and arbitrarily large (plus/minus) $\lambda$-invariant are constructed.
The latter method is more amenable for generalization to our setting.
In Section \ref{section: large lambda invariant}, we fix the elliptic curve $E=32a2$ (Cremona label) and the prime $p=5$.
Using the work of K.~Rubin--A.~Silverberg (see \cite{RS95}), we have an algorithm to produce elliptic curves which are 5-congruent to $E$ with large anticyclotomic $\lambda$-invariant over $K=\Q(\sqrt{-1})$.
Here, we crucially use the formula \eqref{lambda minus lambda}, which relates $\lambda$-invariants of the Selmer groups of elliptic curves in Rubin--Silverberg family.
We illustrate our algorithm via explicit computation.
\section{Basic Notions}
Throughout this article, let $p\geq 5$ be a fixed prime number and $K$ be a number field.
Let $\Kinf$ be any $\Z_p$-extension of $K$.
Two examples of interest are as follows:
\begin{enumerate}
\item $K$ is a number field and $\Kinf=K^{\cyc}$ is the \emph{cyclotomic} $\Z_p$-extension of $K$.
This is the unique $\Z_p$-extension of $K$ contained in $K(\mu_{p^{\infty}})$.
\item $K$ is an imaginary quadratic field and $\Kinf= K^{\ac}$ is an \emph{anticyclotomic} $\Z_p$-extension of $K$.
This is a $\Z_p$-extension of $K$ which is Galois and pro-dihedral over $\Q$.
More generally, $K_\infty$ may be chosen to be an anticyclotomic $\Z_p$-extension over any CM field, $K$.
\end{enumerate}
If $K$ is any number field with signature $(r_1, r_2)$, the Leopoldt's conjecture predicts that the $\Z_p$-rank of the maximal abelian pro-$p$ extension of $K$ unramified away from primes above $p$ is $1+r_2$.
In other words, it predicts that there are $1+r_2$ \emph{independent} $\Z_p$-extensions of $K$.
For an imaginary quadratic field, the cyclotomic and anticyclotomic extensions together generate a $\Z_p^2$-extension, containing infinitely many $\Z_p$-extensions that are unramified at all primes $v\nmid p$ and not Galois over $\Q$.

Set $\Gamma:=\Gal(\Kinf/K)\simeq \Z_p$ and $\Gamma_n:=\Gal(\Kinf/K_n)$.
The Iwasawa algebra $\Lambda$ is the completed group ring $\varprojlim_n \Z_p\llbracket \Gamma/\Gamma_n\rrbracket$.
Choose a topological generator $\gamma\in\Gamma$.
There is an isomorphism $\Lambda \simeq \Z_p\llbracket T \rrbracket $ associating the formal variable $T$ with $\gamma-1$.

For a discrete $p$-primary abelian group $M$, set $M^{\vee}:=\Hom_{\op{cnts}}(M, \Q_p/\Z_p)$ to denote its \emph{Pontryagin dual}.
We say that $M$ is a cofinitely generated (resp. cotorsion) $\Lambda$-module if $M^{\vee}$ is finitely generated (resp. torsion) as a $\Lambda$-module.
For a cofinitely generated and cotorsion $\Lambda$-module $M$, the \emph{Structure Theorem~of $\Lambda$-modules} asserts that $M^{\vee}$ is pseudo-isomorphic to a finite direct sum of cyclic $\Lambda$-modules, i.e., there is a map of $\Lambda$-modules
\[
M^{\vee}\longrightarrow \left(\bigoplus_{i=1}^s \Lambda/(p^{m_i})\right)\oplus \left(\bigoplus_{j=1}^t \Lambda/(h_j(T)) \right)
\]
with finite kernel and cokernel.
Here, $m_i>0$ and $h_j(T)$ is a distinguished polynomial (i.e., a monic polynomial with non-leading coefficients divisible by $p$).
The \emph{characteristic ideal} of $M^\vee$ is (up to a unit) generated by the characteristic element,
\[
f_{M}^{(p)}(T) := p^{\sum_{i} m_i} \prod_j h_j(T).
\]
The $\mu$-invariant of $M$ is defined as the power of $p$ in $f_{M}^{(p)}(T)$.
More precisely,
\[
\mu(M):=\begin{cases}0 & \textrm{ if } s=0\\
\sum_{i=1}^s m_i & \textrm{ if } s>0.
\end{cases}
\]
The $\lambda$-invariant of $M$ is the degree of the characteristic element, i.e.
\[
\lambda(M) := \sum_{j=1}^t \deg h_j(T).
\]
\begin{Remark}\label{splitting remark}
Write $\Sigma$ for the set of primes of $K$ that are finitely decomposed in $\Kinf$.
For $\Kinf=K^{\cyc}$, all primes are finitely decomposed, i.e., $\Sigma$ consists of all primes.

Consider the case when $K$ is an imaginary quadratic field and $\Kinf = \Kan$.
Let $\ell$ be a rational prime and $v|\ell$ be a prime of $K$.
Then, there are finitely many primes $w$ of $\Kan$ lying above $v$ if and only if either $\ell=p$ or $\ell$ splits in $K$ (see for example \cite{Bri07}).
Thus, by the Cheboratev density theorem, $\Sigma$ consists of $50\%$ of the primes of $K$.
\end{Remark}

Henceforth, we fix a number field $K$ and a $\Z_p$-extension $\Kinf$ over $K$.
Let $E_{/K}$ be an elliptic curve with good ordinary reduction at $p$.
We now define our module of interest, the $p$-primary Selmer group over $\Kinf$.
Let $S$ be a finite set of primes $v$ of $K$ \emph{containing} the primes above $p$ and the primes of bad reduction of $E$. 
Denote by $K_S$ the maximal algebraic extension of $K$ which is unramified at the primes $v\notin S$.
For any finite extension $F$ of $K$ such that $F\subseteq K_S$ and a prime $v\in S$, define
\[
J_v(E/F) := \prod_{w|v} H^1\left( F_w, E\right)[p^\infty]
\]
where the product is over all primes $w$ of $F$ lying above $v$.
The \emph{$p$-primary Selmer group over $F$} is defined as follows
\[
\Sel(E/F):=\ker\left\{ H^1\left(K_S/F,E[p^{\infty}]\right)\longrightarrow \bigoplus_{v\in S} J_v(E/F)\right\}.
\]
The $p$-primary Selmer group $\SelpK$ is the direct limit of $\Sel(E/F)$ as $F$ ranges over all finite extensions of $K$ contained in $\Kinf$.
The $\mu$ and $\lambda$-invariants of $\SelpK$ are denoted by $\mupK$ and $\lambdapK$.
We make the following assumption.
\begin{assumption}
\label{fg}
$\SelpK$ is a cotorsion $\Lambda$-module with $\mupK=0$.
\end{assumption}
\begin{Remark}
\label{cotorsion + mu=0 remark}
Such an assumption is expected to hold in the special cases of interest.
\begin{enumerate}
\item Consider the case when $E$ is an elliptic curve defined over $\Q$ and $\Kinf=K^{\cyc}$ for an abelian extension $K/\Q$.
The $p$-primary Selmer group $\SelpK$ is a cotorsion $\Lambda$-module; see \cite{Kat04}.
When $E$ is an elliptic curve for which the residual Galois representation $\overline{\rho}:\Gal(\closure{\Q}/\Q)\rightarrow \GL(\F_p)$ is irreducible, it is conjectured by Greenberg that the $\mu(E/\Q^{\cyc})=0$ (see \cite[Conjecture 1.11]{Gre99}).
\item Consider the case when $E_{/\Q}$ is an elliptic curve of conductor $N$, $K$ is an imaginary quadratic number field, and $\Kinf = \Kan$.
Write $N=N^+ N^-$, where $N^+$ (resp. $N^-$) is divisible by primes that split (resp. are inert) in $K$.
Let $\overline{\rho}:\op{G}_{\Q}\rightarrow \GL(\F_p)$ be the residual Galois representation on the $p$-torsion points $E[p]$.
Suppose that the following conditions are satisfied
\begin{enumerate}
 \item $\overline{\rho}$ is surjective,
 \item if $q$ is a prime such that $q|N^{-}$ and $q\equiv \pm 1 \mod{p}$, then $\overline{\rho}$ is ramified at $q$,
 \item the number of primes dividing $N^-$ is odd.
 \item the $p$-th Fourier coefficient $a_p\not \equiv \pm 1\pmod{p}$.
\end{enumerate}
Then \cite[Theorem~5.3]{PW11} (see also \cite[Remark 1.4]{KPW17}) asserts that Assumption~\ref{fg} holds.
\end{enumerate}
\end{Remark}

\section{Preliminary Results}
As before, let $p\geq 5$ be a prime number and $K$ be a number field.
Let $L$ be a finite Galois extension of $K$ with the order of the Galois group $G=\Gal(L/K)$ a power of $p$.
Let $\Kinf$ be a $\Z_p$-extension of $K$ and $\Linf=\Kinf\cdot L$.
The field diagram is drawn below
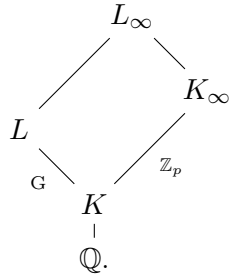
\begin{equation*}
\begin{tikzpicture}[node distance = 1.5cm, auto]
 \node(Q) {$\Q$.};
 \node (K) [node distance= 0.75cm, above of =Q]{$K$};
 \node (L) [node distance= 1cm, above of=K, left of=K] {$L$};
 \node (Kinf) [above of=K, right of =K] {$\Kinf$};
 \node (Linf) [above of=L, right of=L] {$\Linf$};
 \draw[-] (Q) to node {} (K);
 \draw[-] (K) to node {{\tiny G}} (L);
 \draw[-] (K) to node [swap]{{\tiny $\mathbb{Z}_p$}} (Kinf);
 \draw[-] (L) to node {} (Linf);
 \draw[-] (Kinf) to node {} (Linf);
 \end{tikzpicture}\hspace{1.0cm}
 \end{equation*}
In the two examples of special interest, we know the following.
\begin{enumerate}
 \item When $K$ is a number field and $\Kinf=K^{\cyc}$ is the \emph{cyclotomic} $\Z_p$-extension of $K$, $L_\infty$ is identified with the cyclotomic $\Z_p$-extension of $L$, namely $L^{\cyc}$.
 \item When $K$ is an imaginary quadratic field and $\Kinf= K^{\ac}$ is an \emph{anticyclotomic} $\Z_p$-extension of $K$, the $\Z_p$-extension, $\Linf$ over $L$ is a non-cyclotomic extension (but it is not an anticyclotomic extension).
\end{enumerate}

In this section, we record some preliminary results required in the proof of Theorem~\ref{theorem1}.
For the remainder of this discussion, we assume the following.
\begin{assumption}\label{elliptic curve assumptions}
\begin{enumerate}
 \item $E$ has good ordinary reduction at the primes of $K$ above $p$.
 \item At least one of the following hold
 \begin{enumerate}
 \item $E$ does not have complex multiplication, 
 \item $E(L)[p]=0$.
 \end{enumerate}
\end{enumerate}
\end{assumption}
\noindent The assumption of $E(L)[p]=0$ is weaker than insisting that the residual representation $\overline{\rho}:\Gal(\closure{L}/L)\rightarrow \GL(\F_p)$ on $E[p]$ is irreducible.

\begin{Lemma}
\label{wo CM, finite over Zp extension}
Let $K$ be a number field and $E_{/K}$ be an elliptic curve without complex multiplication.
Let $\cL/K$ be any algebraic extension which is Galois over $K$.
Then, either
\begin{enumerate}[label=(\alph*)]
 \item $E(\cL)[p^\infty]=E[p^{\infty}]$ \emph{or}
 \item $E(\cL)[p^\infty]$ is finite.
\end{enumerate}
\end{Lemma}

\begin{proof}
Set $M = E(\cL)[p^\infty]$ and $T_p(M) = \varprojlim_n M[p^n]$.
Note that $V_p(M) := T_p(M)\otimes_{\Z_p} \Q_p$ is a submodule of $V_p(E) := T_p(E)\otimes_{\Z_p} \Q_p$ where $T_p(E)$ is the (usual) Tate module of $E$.
Since $E$ is an elliptic curve without complex multiplication, by a theorem of Serre (see \cite[IV, Theorem~2.1(a)]{Ser97}) we know that $V_p(E)$ (hence $V_p(M)$) is irreducible.
Thus, $V_p(M)$ is either trivial or equal to $V_p(E)$.
Since $\dim V_p(M)=\corank_{\Z_p} M$, it follows that $V_p(M)=0$ if and only if $M$ is finite.
But, $E[p^{\infty}]$ is cofree, so $V_p(M)=V_p(E)$ if and only if $M=E[p^{\infty}]$.
This completes the proof of the lemma.
\end{proof}

\begin{Cor}\label{finiteness corollary}
Let $K$ be a number field and $K_{\infty}$ be a $\Z_p$-extension of $K$.
Let $E_{/K}$ be an elliptic curve.
Assume that either of the following conditions are satisfied
\begin{enumerate}
 \item $E$ does not have complex multiplication,
 \item $E(K)[p]=0$.
\end{enumerate}
Then, $E(K_{\infty})[p^{\infty}]$ is finite.
\end{Cor}

\begin{proof}
\begin{enumerate}
 \item When $E$ does not have complex multiplication, it follows from Serre's big image theorem that the field $K(E[p^{\infty}])$ is not a solvable extension of $K$.
 Therefore, $E(K_{\infty})[p^{\infty}]\neq E[p^{\infty}]$.
 The assertion follows from Lemma~\ref{wo CM, finite over Zp extension}.
 \item Since $\Gamma\simeq \Z_p$ is pro-$p$, it follows that $E(K_{\infty})[p]=0$ (see for example \cite[I.6.13]{NSW}).
 Hence, $E(K_{\infty})[p^{\infty}]=0$.
\end{enumerate} 
\end{proof}

The following lemma is an analogue of \cite[Proposition~2.3]{HM99} in the current setting (see also \cite[Proposition~5.30]{How_thesis}).
In proving the following lemma, we crucially use the hypothesis introduced in Assumption~\ref{elliptic curve assumptions}(1) that $p$ has good ordinary reduction at $p$.
Note that Lemma~\ref{wo CM, finite over Zp extension} and Corollary~\ref{finiteness corollary} do not require this assumption.

\begin{Lemma}\label{analogue of HM 2.3}
Let $E$ be an elliptic curve with good ordinary reduction at the primes of $K$ above $p$.
Suppose that $\SelpK$ is $\Lambda$-cotorsion and $E(\Kinf)[p^\infty]$ is finite.
Then the map
\begin{equation}\label{surjectivity equation}
H^1\left( K_S/\Kinf, E[p^\infty]\right) \longrightarrow \bigoplus_{v\in S} J_v(E/\Kinf)
\end{equation}
is surjective.
Furthermore, $H^2\left( K_S/ \Kinf, E[p^\infty]\right)=0$.
\end{Lemma}

\begin{proof}
Recall the Cassels--Poitou--Tate sequence (see for example \cite[p.~9]{GCEC})
\begin{align*}
0 \rightarrow \SelpK &\rightarrow H^1\left(K_S/\Kinf, E[p^\infty] \right) \xrightarrow{\theta} \bigoplus_{v\in S} J_v(E/\Kinf)\\
&\rightarrow \left(\varprojlim_{n} \mathfrak{S}_p(E/K_n)\right)^{\vee} \rightarrow H^2\left(K_S/\Kinf, E[p^\infty] \right) \rightarrow 0.
\end{align*} 
Here $\mathfrak{S}_p(E/K_n) := \varprojlim_m \Sel_{p^m}(E/K_n)$, and the inverse limit is taken with respect to the maps induced by multiplication by $p$.
The inverse limit of $\mathfrak{S}_p(E/K_n)$ is taken with respect to the corestriction map.
In order to prove the result, it suffices to show that $\varprojlim_{n} \mathfrak{S}_p(E/K_n)=0$.

Consider the exact sequence (see \cite[Lemma~1.8]{GCEC})
\[
0 \rightarrow E(K_n)[p^\infty] \rightarrow \mathfrak{S}_p(E/K_n) \rightarrow \Hom_{\Z_p}\left( \Sel(E/K_n)^\vee, \Z_p\right) \rightarrow 0.
\]
Taking inverse limits, we have an exact sequence
\[
\varprojlim_n E(K_n)[p^\infty]\rightarrow \varprojlim_n \mathfrak{S}_p(E/K_n) \rightarrow \varprojlim_n \Hom_{\Z_p}\left( \Sel(E/K_n)^\vee, \Z_p\right).
\]
Here, the inverse limit $\varprojlim_n E(K_n)[p^\infty]$ is taken with respect to norm maps.
Since $E(\Kinf)[p^\infty]$ finite, it follows that $\varprojlim_n E(K_n)[p^\infty] =0$.
Therefore, we have an injection 
\[ \varprojlim_n \mathfrak{S}_p(E/K_n) \hookrightarrow \varprojlim_n \Hom_{\Z_p}\left( \Sel(E/K_n)^\vee, \Z_p\right).
\]
Set $\Gamma_n = \Gal\left( \Kinf/K_n\right)$.
By an application of the Control Theorem, (see \cite[Theorem~4.1]{Gre_PCMS} or \cite[Theorem~1]{Gre03}) we have an injection
\[
\varprojlim_n \mathfrak{S}_p(E/K_n) \hookrightarrow \varprojlim_n \Hom_{\Z_p}\left( \SelpK^{\vee}_{\Gamma_n}, \Z_p\right).
\]
The above injection requires that $\ker\left( \Sel(E/K_n) \rightarrow \Sel(E/\Kinf)^{\Gamma_n}\right)$ is finite (but not necessarily bounded).
For a $\Lambda$-module $M$, set $M^\iota$ for the $\Lambda$-module with the same underlying $\Z_p$-module and inverse $\Gamma$-action.
By \cite[\S 2 Lemme 4(i)]{PR87}), there is an isomorphism
\[
\varprojlim_n \Hom_{\Z_p}\left( \SelpK^{\vee}_{\Gamma_n}, \Z_p\right) \simeq \Hom_{\Lambda}\left( \SelpK^\vee, \Lambda\right)^{\iota}.
\]
Since $\SelpK$ is a cotorsion $\Lambda$-module, it follows that $\Hom_{\Lambda}\left( \SelpK^\vee, \Lambda\right)^{\iota}=0$.
Therefore, $\varprojlim_n \mathfrak{S}_p(E/K_n)=0$.
The assertion of the lemma follows.
\end{proof}

\section{Generalizations of Kida's formula}
In this section, we state a generalization of Kida's formula which applies to more general $\Z_p$-extensions.
Recall that $p\geq 5$ is a fixed prime number.
Let $K$ be a number field and $\Kinf$ be any $\Z_p$-extension over $K$ such that all primes above $p$ are ramified in $\Kinf$.
Let $L$ be a finite Galois extension of $K$ with Galois group $G=\Gal(L/K)$ of $p$-power order.
Let $E_{/K}$ be an elliptic curve such that the Assumption~\ref{elliptic curve assumptions} holds.
Recall that $S$ is a finite set of primes of $K$ \emph{containing} the primes above $p$ and the primes of bad reduction of $E$.
For our discussion on generalized Kida-type formula, we choose the set $S$ to be the primes $v$ in $K$ satisfying \emph{at least one} of the following conditions
\begin{enumerate}
 \item $v|p$,
 \item $E$ has bad reduction at $v$,
 \item $v$ ramifies in the extension $L/K$.
\end{enumerate}
Let $\Linf$ denote the composite $\Kinf \cdot L$.
We define two sets of primes $\cP_1$ and $\cP_2$ in $L_\infty$, which play a crucial role in the generalization of Kida's formula.
\begin{Defi}
\label{sets Pi}
For $i=1,2$, the set $\cP_i$ consists of primes $w'\nmid p$ in $\Linf$ such that
\begin{enumerate}
 \item $w'\in \cP_1$ if $E$ has split multiplicative reduction at $w'$,
 \item $w'\in \cP_2$ if \emph{all} of the following conditions are satisfied
 \begin{enumerate}
 \item $w'$ is ramified above a prime $v^\prime$ in $\Kinf$,
 \item $E$ has good reduction at $w'$
 \item $E(L_{\infty,w'})$ possesses a point of order $p$.
 \end{enumerate}
\end{enumerate}
\end{Defi}
Recall that $\Sigma$ denotes the set of primes of $K$ that are finitely decomposed in $\Kinf$.
We introduce an important assumption.
\begin{assumption}
\label{finitely decomposed}
Assume that each prime $w'\in \cP_1\cup \cP_2$ lies above a prime $v\in \Sigma$.
\end{assumption}
By assumption, any prime $w'\in \cP_1\cup \cP_2$ must lie above a prime $v\in S\cap \Sigma$.
The above assumption guarantees that the sets $\cP_i$ are finite for $i=1,2$.
Our choice of $\cP_1$ is the same as that in \cite{HM99}, and $\cP_2$ bears the additional condition that the primes $w'\in \cP_2$ are ramified over $\Kinf$.
This does not make an actual difference in terms in the formula relating the $\lambda$-invariants.

For any algebraic extension $\cL$ of $K$, and a set of primes $T$ of $K$, we write $T(\cL)$ to denote the set of primes of $w$ of $\cL$ such that $w|v$ for primes $v\in T$.
For a prime $v$ of $K$, we write $v(\cL)$ for the set of primes $w|v$ of $\cL$.

We record our first main result; its proof will occupy this section and the next.
\begin{Th}\label{main theorem}
For each prime $w'$ of $\Linf$, write $e_{w'}(\Linf/\Kinf)$ for the ramification index of $w'$ for the extension $\Linf/\Kinf$.
The $\lambda$-invariant of $\SelpL$ is given by the formula
\[\lambdapL=[L_{\infty}:K_{\infty}] \lambdapK+\sum_{w'\in \cP_1} \left(e_{w'}-1\right)+\sum_{w'\in \cP_2} 2\left(e_{w'}-1\right).\]
\end{Th}

Since $G$ is a $p$-group, it is solvable.
Let $L_0\subset L_1\subset \cdots \subset L_M$ be a filtration of $L$ such that $L_{i+1}$ is Galois over $L_i$ with Galois group $\Gal(L_{i+1}/L_i)\simeq \Z/p\Z$.
It is easy to show that it suffices to prove Theorem~\ref{main theorem} when $\#G=p$ (see for example \cite[Lemma~3.5]{Mat00}).
Without loss of generality, we assume from here on that $G\simeq \Z/p\Z$. 


\begin{Lemma}\label{boring lemma}
Let $K$ be a number field and $K_{\infty}$ be a $\Z_p$-extension of $K$.
Let $v\nmid p$ be a prime of $K$ and $v'|v$ a prime of $K_{\infty}$.
The following assertions hold
\begin{enumerate}
 \item if $v$ is finitely decomposed in $K_{\infty}$, then the localization $K_{\infty, v'}=K_v^{\cyc}$.
\item\label{lemma43c2} if there exists a Galois $p$-extension of $K_v$, then $K_{v}$ contains $\mu_p$.
\end{enumerate}
\end{Lemma}

\begin{proof}
\begin{enumerate}
 \item Since $v$ is finitely decomposed in $\Kinf$, we see that $K_{\infty,v'}$ is a $\Z_p$-extension of $K$.
 Since $v\nmid p$, 
 we know that $v$ is unramified in $\Kinf$ (see \cite[Proposition~13.2]{Was97}).
 By local class field theory, there is a unique unramified pro-extension of $K_v$ and therefore $K_{\infty,v'}=K_v^{\cyc}$.
 \item This statement is a direct consequence of local class field theory.
\end{enumerate}
\end{proof}

We keep the notation introduced at the start of this section.
Let $v$ (resp. $w$) be primes above $\ell$ in $K$ (resp. $L$) such that $w|v$.
Let $w'| w$ be a prime of $\Linf$ and $v':=w'_{\restriction \Kinf}$.
The diagram below depicts the labelling of primes
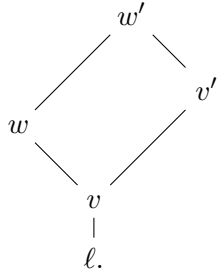
\begin{center}
 \begin{tikzpicture}[node distance = 1.5cm, auto]
 \node(Q) {$\ell.$};
 \node (K) [node distance = 0.75cm, above of =Q]{$v$};
 \node (L) [node distance = 1.0cm, above of=K, left of=K] {$w$};
 \node (Kinf) [above of=K, right of =K] {$v'$};
 \node (Linf) [above of=L, right of=L] {$w'$};
 \draw[-] (Q) to node {} (K);
 \draw[-] (K) to node {} (L);
 \draw[-] (K) to node {} (Kinf);
 \draw[-] (L) to node {} (Linf);
 \draw[-] (Kinf) to node {} (Linf);
 \end{tikzpicture}
 \end{center}
\begin{Lemma}\label{ep infinity}
Let $v\in S\setminus S_p$ and $w'\notin \mathcal{P}_1\cup \mathcal{P}_2$ be a prime of $L_{\infty}$ above $v$. 
Set $v'$ and $w$ be as in the diagram above.
Then at least one of the following assertions hold.
\begin{enumerate}
 \item\label{c1lemma51} $L_{\infty,w'}=K_{\infty,v'}$,
 \item\label{c2lemma51} $E(L_{\infty,w'})[p^{\infty}]=0$.
\end{enumerate}
\end{Lemma}
\begin{proof}
Suppose that $L_{\infty,w'}\neq K_{\infty,v'}$, then it follows that $L_w$ is a nontrivial $p$-extension of $K_v$ which is Galois over $K_v$.
Therefore by local class field theory, $\mu_p$ is contained in $K_v$.
Since $v\in S\setminus S_p$, either $v$ is ramified in $L$ or $E$ has bad reduction at $v$.
If $E$ has good reduction at $v$, then $v$ is ramified in $L$.
Since $v\nmid p$, it is unramified in $\Kinf$ (see \cite[Proposition~13.2]{Was97}).
Therefore, $L_{\infty,w'}/ K_{\infty,v'}$ is a ramified extension.
Recall that $\mathcal{P}_2$ is the set of primes $w'$ of $L_{\infty}$ at which \emph{all} of the following assumptions are satisfied
\begin{enumerate}
 \item $L_{\infty,w'}/ K_{\infty,v'}$ is a ramified extension,
 \item $E$ has good reduction at $w'$ and $E(L_{\infty,w'})[p^{\infty}]\neq 0$.
\end{enumerate}
Note that since $E$ has good reduction at $v$, it also has good reduction at $w'$.
Since $L_{\infty,w'}/ K_{\infty,v'}$ is a ramified extension and $w'\notin \mathcal{P}_2$, it follows that $E(L_{\infty,w'})[p^{\infty}]=0$.

When $E$ has bad reduction at $v$, there are two possible subcases
\begin{enumerate}[label=(\roman*)]
 \item $E$ has good reduction at $w'$ \emph{or}
 \item $E$ has bad reduction at $w'$.
\end{enumerate}
If $E$ has good reduction at $w'$, it is easy to see from the Neron-Ogg-Shafarevich criterion that $L_{\infty,w'}/ K_v$ is ramified.
Since $v$ is unramified in $\Kinf$, it follows that $L_{\infty,w'}/ K_{\infty,v'}$ is a ramified extension.
In this setting the same argument as in the previous case applies.

One is left with the case when $E$ has bad reduction at $w'$.
Since $w'\notin \mathcal{P}_1$, we have by assumption that $E$ either has non-split multiplicative reduction or additive reduction at $w'$.
In this setting, \cite[Proposition~5.1 (iii)]{HM99} applies.
Since $p\geq 5$, it follows that $E(L_{\infty,w'})[p^{\infty}]=0$.
This result may be applied since $\mu_p$ is contained in $K_v$, which is one of the required hypotheses.
This has been shown to hold in the beginning of the proof.
\end{proof}

\begin{Lemma}\label{kergammav for non p}
Let $v$ be any prime is $S \setminus S_p$.
Then the kernel of the map
\[\gamma_v: J_v(E/\Kinf)\rightarrow J_v(E/\Linf)\] is finite.
\end{Lemma}

\begin{proof}
Recalling that $v\nmid p$ may be infinitely decomposed in $\Kinf/K$, i.e., the set $v(\Kinf)$ may be infinite.
At each prime $v' \in v(\Kinf)$, choose a prime $w'$ of $\Linf$ such that $w'|v'$.
It follows from the inflation-restriction sequence that 
\begin{equation}
\label{ker gammav}
\ker \gamma_v \subseteq \prod_{v'\in v(\Kinf)}H^1\left(L_{\infty,w'}/K_{\infty,v'}, E(L_{\infty,w'})\right)[p^\infty], 
\end{equation}
the right hand side of which may possibly be an infinite product.
Since $v'\nmid p$, we know that (see for example \cite[Theorem~2.4(i)]{Gre_PCMS})
\[
E(K_{\infty,v'})\otimes \Q_p/\Z_p=0.\]
The following isomorphism is immediate from the Kummer sequence,
\begin{equation}
\label{kummer sequence iso}
H^1\left(L_{\infty,w'}/K_{\infty,v'}, E(L_{\infty,w'})\right)[p^\infty] \simeq H^1\left(L_{\infty,w'}/K_{\infty,v'}, E(L_{\infty,w'})[p^\infty]\right).
\end{equation}
Since $E(L_{\infty,w'})[p^\infty]$ is a cofinite $\Z_p$-module and $G= \Gal(L/K)$ is finite, the right hand side of \eqref{kummer sequence iso} is a finite group.
To prove the lemma, we consider two cases\newline
\emph{Case 1:} $v\in \Sigma$.
By definition of $\Sigma$, the set $v(\Kinf)$ is finite.
The result is immediate.\newline
\emph{Case 2:} $v\notin \Sigma$.
Assumption~\ref{finitely decomposed} implies that $w'\notin \mathcal{P}_1\cup \mathcal{P}_2$.
By Lemma~\ref{ep infinity} we have that
\begin{equation}\label{trivial H1 in non p case}
H^1\left(L_{\infty,w'}/K_{\infty,v'}, E(L_{\infty,w'})[p^\infty]\right)=0.
\end{equation}
Therefore, $\ker\gamma_v$ is zero when $v\nmid \Sigma$.
\end{proof}

\begin{Lemma}\label{kergammav for p}
Let $v\in S_p$ and consider the restriction map 
\[\gamma_v: J_v(E/\Kinf)\rightarrow J_v(E/\Linf).\] 
The kernel of $\gamma_v$ is trivial.
\end{Lemma}

\begin{proof}
Recall that each of the primes $v\in S_p$ is assumed to be ramified in $\Kinf$.
It follows that if $v'|v$ is a prime in $\Kinf$, the extension $K_{\infty, v'}/K_v$ is \textit{deeply ramified} in the sense of Coates--Greenberg (see \cite[p.~130]{CG96}).
For each $v'\in v(\Kinf)$, choose a prime $w'$ of $\Linf$ such that $w'|v'$.
We have the isomorphism
\[
\ker \gamma_v \simeq \bigoplus_{v'\in v(\Kinf)}H^1\left(L_{\infty,w'}/K_{\infty,v'}, E(L_{\infty,w'})\right)[p^\infty].
\]
We will show that $H^1\left(L_{\infty,w'}/K_{\infty,v'}, E(L_{\infty,w'})\right)=0$ for $v'\in v(\Kinf)$.

Let $\cF$ denote the formal group attached to $E$.
Let $\mathfrak{m}_{w'}$ denote the maximal ideal of (the ring of integers of) $L_{\infty,w'}$ and $l_{w'}$ be the residue field.
Let $\widetilde{E}$ denote the reduction of $E$ modulo $\mathfrak{m}_{w'}$.
Consider the short exact sequence of $G$-modules (see \cite[p.~124]{Sil09})
\[
0 \rightarrow \cF(\mathfrak{m}_{w'}) \rightarrow E(L_{\infty,w'}) \rightarrow \widetilde{E}(l_{w'}) \rightarrow 0.
\]
Since $K_{\infty,v'}$ is a deeply ramified extension, it follows from \cite[Theorem~3.1]{CG96} that
\[
H^i\left( G, \cF(\mathfrak{m}_{w'})\right)=0 \quad \textrm{ for }i=1,2.
\]
A connected commutative algebraic group over a finite
field is cohomologically trivial (see result of S.~Lang in \cite[p.~204]{Maz72}).
Therefore,
\[
H^i\left( G, \widetilde{E}(l_{w'})\right)=0 \quad \textrm{ for }i=1,2.
\]
The lemma is now immediate.
\end{proof}

\begin{Proposition}\label{finite kernel and cokernel}
With setting as above, the kernel and cokernel of the restriction map 
\[\alpha: \SelpK\rightarrow \SelpL^{G}\]are finite.
\end{Proposition}
\begin{proof}
The map $\alpha$ fits into a diagram
\[
{\begin{tikzcd}[column sep = small, row sep = large]
0\arrow{r} & \SelpK \arrow{r}\arrow{d}{\alpha} & H^1\left( K_S/\Kinf, E[p^\infty]\right)\arrow{r} \arrow{d}{\beta} & \bigoplus_{v\in S} J_v(E/\Kinf)\arrow{r} \arrow{d}{\gamma} & 0\\
0\arrow{r} & \SelpL^G \arrow{r} & H^1\left( K_S/\Linf, E[p^\infty]\right)^G \arrow{r} &\bigoplus_{v\in S} J_v(E/\Linf).
\end{tikzcd}}\]
The map $\gamma$ decomposes into a direct sum of local maps
\[\gamma_v: J_v(E/\Kinf)\rightarrow J_v(E/\Linf).\]
By the snake lemma, it suffices to show the following
\begin{enumerate}
 \item the kernel and cokernel of $\beta$ are both finite,
 \item the kernel of each map $\gamma_v$ is finite (resp. trivial) when $v|\ell$ such that $\ell$ is finitely decomposed (resp. $\ell$ splits completely) in $\Kinf/K$.
\end{enumerate} 
By the inflation-restriction sequence, the map $\beta$ fits in an exact sequence
\[\begin{split}0 &\rightarrow H^1\left(G, E(\Linf)[p^{\infty}]\right)\rightarrow H^1\left( K_S/\Kinf, E[p^\infty]\right)\xrightarrow{\beta} H^1\left( K_S/\Linf, E[p^\infty]\right)^G \\
&\rightarrow H^2\left(G, E(\Linf)[p^{\infty}]\right)\rightarrow \dots\end{split}.\]
Since $E(\Linf)[p^{\infty}]$ is of cofinite type, the cohomology groups $H^i\left(G, E(L_{\infty})[p^{\infty}]\right)$ are finite for $i=1,2$.
Therefore, the kernel and cokernel of $\beta$ are both finite.
The result follows from the analysis of the local kernel (see Lemmas \ref{kergammav for non p} and \ref{kergammav for p}).
\end{proof}
The next corollary is a generalization of \cite[Corollary~3.4]{HM99}.

\begin{Cor}
\label{cor: mu=0 for L}
Let $E_{/K}$ be an elliptic curve such that $\SelpK$ is a $\Lambda$-cotorsion module with $\mupK=0$.
If $L/K$ is a Galois extension such that Lemma~\ref{finite kernel and cokernel} holds, then $\SelpL$ is a $\Lambda$-cotorsion module with $\mupL=0$.
\end{Cor}

\begin{proof}
Note that $\SelpK$ is a cotorsion $\Lambda$-module with $\mu=0$ if and only if it is cofinitely generated as a $\Z_p$-module.
It follows from Proposition~\ref{finite kernel and cokernel} that $\SelpL^G$ is also a cofinitely generated $\Z_p$-module.
Equivalently, the $G$-coinvariant of $\SelpL^{\vee}$ is finitely generated as a $\Z_p$-module.
Since $G\simeq \Z/p\Z$, the group ring $\Z_p[G]$ is local.
By Nakayama's Lemma, $\SelpL^{\vee}$ is finitely generated as a $\Z_p[G]$-module and hence finitely generated as a $\Z_p$-module.
This completes the proof of the lemma.
\end{proof}

Since Assumption~\ref{fg} holds, it follows from Proposition~\ref{finite kernel and cokernel} and Corollary~\ref{cor: mu=0 for L} that
\begin{align*}
 \lambdapK &= \corank_{\Z_p} \SelpK = \corank_{\Z_p} \SelpL^{G},\\
 \lambdapL &= \corank_{\Z_p} \SelpL.
\end{align*}

\begin{Defi}
Let $G$ be a $p$-group and $M$ be a divisible $\Z_p[G]$-module of cofinite type.
Then the \emph{Herbrand quotient} of $M$ is defined as
\[
h_G(M) = \frac{\# H^2(G,M)}{\# H^1(G,M)}.
\]
\end{Defi}
\noindent When $G=\Z/p\Z$ and $M = \SelpL$, it is straightforward to see from the work of Iwasawa that (cf. \cite[p.~589]{HM99})
\begin{equation}\label{lambdaL to lambdaK}
 \lambdapL = p\lambdapK + (p-1)\ord_p \left( h_G\left(\SelpL\right)\right).
\end{equation}

\section{Herbrand Quotient Calculations}

Let $E_{/K}$ be an elliptic curve.
Throughout this section, we suppose that Assumptions \ref{fg}, \ref{elliptic curve assumptions} and \ref{finitely decomposed} are satisfied.
It follows from \eqref{lambdaL to lambdaK} that to complete the proof of Theorem~\ref{main theorem}, we need to calculate the Herbrand quotient $h_G\left(\SelpL\right)$.
We carry out the calculation in this section, thereby completing the proof of the theorem.

\subsection{Simplification of the Herbrand quotient}
Corollary~\ref{cor: mu=0 for L} asserts that $\SelpL$ is $\Lambda$-cotorsion with $\mupL=0$.
Next, by Corollary~\ref{finiteness corollary} that $E(L_\infty)[p^\infty]$ is finite.
Using an argument identical to Lemma~\ref{analogue of HM 2.3}, the Cassels--Poitou--Tate sequence becomes
\[
0 \rightarrow \SelpL \rightarrow H^1\left(K_S/L_{\infty}, E[p^\infty] \right) \rightarrow \bigoplus_{v\in S} J_v(E/L_\infty) \rightarrow 0.
\]
It follows that
\[
h_G\left(\SelpL\right) = \frac{h_G\left(H^1(K_S/L_{\infty}, E[p^\infty] )\right)}{ \prod_{v\in S} h_G\left(J_v(E/L_\infty)\right)}.
\]
Let $v\in S\setminus S_p$.
Write $v^*(\Linf)$ for the set of primes $w'\in v(\Linf)$ satisfying both conditions
\begin{enumerate}
 \item $w'$ is a ramified prime above $v^\prime\in v(\Kinf)$.
 \item $w'\in \cP_1\cup \cP_2$.
\end{enumerate}
It follows from Assumption~\ref{finitely decomposed} that $v^*(\Linf)$ is finite.
Since $L_{\infty}$ is a Galois extension of $K$, it follows that either $v^*(L_{\infty})=\emptyset$ or $v^*(L_{\infty})=v(L_{\infty})$.

\begin{Lemma} The following assertions hold.
\begin{enumerate}
 \item The Herbrand quotient \[h_G\left(H^1\left(K_S/L_{\infty}, E[p^\infty] \right)\right)= h_G\left(E(L_{\infty})[p^{\infty}]\right)=1.\]
 \item The local Herbrand quotient is expressed as follows \[h_G\left(J_v(E/L_\infty)\right) =\prod_{w'\in v^*(\Linf)} h_G\left(E(L_{\infty,w'})[p^{\infty}]\right).\]
 In the above, the empty product is taken to be $1$.
\end{enumerate}

\end{Lemma}
\begin{proof}
\begin{enumerate}
\item This assertion follows from \cite[Lemma~4.1]{HM99}.
\item By \cite[Lemmas 4.2 and 4.3]{HM99},
\[h_G\left(J_v(E/L_\infty)\right) =\prod_{w'\in v(\Linf)} h_G\left(E(L_{\infty,w'})[p^{\infty}]\right).\]
If $w'|p$, the Herbrand quotient $h_G\left(E(L_{\infty,w'})[p^{\infty}]\right)=1$ by \cite[Lemma~4.3]{HM99}, which generalizes verbatim to our setting.
Next, assume that $v\in S\setminus S_p$.
To prove the claim, we show that $h_G\left(E(L_{\infty,w'})[p^{\infty}]\right)=1$ for $w'\notin v^*(L_{\infty})$.
When $w'\notin v^*(L_{\infty})$, either of the following conditions is satisfied
\begin{enumerate}
 \item $L_{\infty,w'}/K_{\infty,v'}$ is unramified,
 \item $w'\notin \mathcal{P}_1\cup \mathcal{P}_2$.
\end{enumerate} 
If $w'\notin \mathcal{P}_1\cup \mathcal{P}_2$, then it follows from Lemma~\ref{ep infinity} that \[h_G\left(E(L_{\infty,w'})[p^{\infty}]\right)=0.\]
Therefore, assume that $w'\in \mathcal{P}_1\cup \mathcal{P}_2$.
Thus, $L_{\infty,w'}/K_{\infty,v'}$ must be an unramified extension.
By Assumption~\ref{finitely decomposed}, since $w'\in \mathcal{P}_1\cup \mathcal{P}_2$, the prime $v\in \Sigma$ and is finitely decomposed in $\Kinf$.
Since $v\nmid p$, it follows that $K_{\infty,v'}=K_v^{\op{cyc}}$, the unique unramified pro-$p$ extension of $K_v$.
Since $L_{\infty,w'}/K_{\infty,v'}$ is unramified, it follows that
\[
L_{\infty,w'}=K_{\infty,v'}=K_v^{\op{cyc}}.
\]
In this case, it is clear that the Herbrand quotient \[h_G\left(E(L_{\infty,w'})[p^{\infty}]\right)=1\] since the group $\op{Gal}(L_{\infty,w'}/K_{\infty,v'})$ is trivial.
\end{enumerate}
\end{proof}

\begin{Cor}\label{local herbrand}
With notation as above,
\[\ord_p \left(h_G\left(\SelpL\right)\right)=-\sum_{v\in S}\left(\sum_{w'\in v^*(\Linf)} \ord_p\left(h_G(E(L_{\infty,w'})[p^{\infty}])\right)\right).\]
\end{Cor}

\begin{proof}[Proof of Theorem~\ref{main theorem}]
Recall that it suffices to prove the theorem when $\# G=p$.
It is known that (see \cite[Corollary~5.2]{HM99}) 
\[\ord_p\left(h_G\left(E\left(L_{\infty,w'}\right)[p^{\infty}]\right)\right)=\begin{cases} -1\text{ if }w'\in \mathcal{P}_1,\\
-2\text{ if }w'\in \mathcal{P}_2.
\end{cases}\]
The result follows from \eqref{lambdaL to lambdaK} and Corollary~\ref{local herbrand}.
\end{proof}

\section{Elliptic Curves with Congruent Galois representations}
\label{section congruent galois representations}
In this section, we study the $p$-primary Selmer groups of $p$-congruent elliptic curves over (general) $\Z_p$-extensions.
We show that for two $p$-congruent elliptic curves, the $p$-primary Selmer group of one is a finitely generated $\Z_p$-module if and only if the same is true for the other elliptic curve (see Theorem~\ref{resisothm} for the precise statement).
Using this result, we can compare the $\lambda$-invariants of these Selmer groups (see Proposition~\ref{prop: compare lambdaS0}).
This is accomplished by introducing \emph{imprimitive Selmer groups} for general $\Z_p$-extensions.
Proposition~\ref{prop: compare lambdaS0} is the key tool which allows us to construct elliptic curves with large $\lambda$-invariants in Section \ref{section: large lambda invariant}.
\par 
Assume throughout this section that $p\geq 5$.
Consider elliptic curves $E_1$ and $E_2$ defined over \emph{any} number field $K$ such that
\begin{enumerate}
 \item both $E_1$ and $E_2$ have good ordinary reduction at the primes $v\in S_p$,
 \item the elliptic curves are $p$-congruent, i.e., there is a $\Gal(\closure{K}/K)$-equivariant isomorphism $E_1[p]\simeq E_2[p]$,
 \item $E_1(K)[p]=0$ (or equivalently $E_2(K)[p]=0$).
 Recall that this condition is automatically satisfied when the residual representation is irreducible.
\end{enumerate}

Since $E_i[p]=0$, by an application of Nakayama's lemma, $E_i(\Kinf)[p^{\infty}]=0$.
For $i=1,2$, let $N_i$ be the conductor of $E_i$ and let $\closure{N}_i$ be the Artin conductor of the residual representation $E_i[p]$.
Note that $\closure{N}_i$ divides $N_i$ for $i=1,2$.
Since $E_1$ and $E_2$ are $p$-congruent, we have that $\closure{N}_1=\closure{N}_2$.
Let $\Kinf$ be \emph{any} $\Z_p$-extension of $K$ and for $i=1,2$ denote by $\Sel(E_i/\Kinf)$ the $p$-primary Selmer group of $E_i$ over $\Kinf$.

\begin{Defi}
\label{def of Omega0}
Recall that we denote by $\mathfrak{T}$ the set of primes $v$ of $K$ at which either $E_1$ or $E_2$ (or both) have bad reduction.
For $i=1,2$, let $\Omega(E_i)$ be the set of primes $v\in \mathfrak{T}$ satisfying \emph{either} of the following conditions.

\begin{enumerate}
    \item\label{def 6.1 c1} The elliptic curve $E_i$ has good reduction at $v$ and \[ \#E_i(\F_v)[p]\neq 0,\]where $\F_v$ is the residue field at $v$.
    \item\label{def 6.1 c2} The elliptic curve $E_i$ has bad reduction at $v$.
    Furthermore, setting $K':=K(\mu_p)$, the elliptic curve $E_i$ has split multiplicative reduction at all primes $w\mid v$ of $K'$.
\end{enumerate}
\end{Defi}
\begin{Remark}
Note that the second condition above implies that if $K_v$ contains $\mu_p$, then, $E_i$ has split multiplicative reduction at $v$.
\end{Remark}

Let $\Omega_0$ denote the set $\Omega(E_1)\cup \Omega(E_2)$.
Since $E_i$ have good reduction at $v\in S_p$, note that $S_p\cap \Omega_0 = \emptyset$.
Write $\Omega$ for the set $\Omega_0\cup S_p$ and $\widetilde{\Omega}$ be the set of primes $v$ for which $v\in S_p$ or one of the elliptic curves $E_1$ or $E_2$ has bad reduction at $v$.
Thus, $\widetilde{\Omega}$ \emph{contains} the set $\Omega$.
Recall that $\Sigma$ is the set of primes of $K$ that are finitely decomposed in $\Kinf$.

For the remainder of the section we make the following assumption for $E_1$ and $E_2$.
\begin{assumption}
\label{finitely decomposed assumption}
For $i=1,2$, assume that $\Omega(E_i)\subset \Sigma$.
\end{assumption}
Throughout, we write $E$ to denote either of the elliptic curves, $E_1$ or $E_2$.
\begin{Lemma}\label{no finite index lemma}
Suppose that $\SelpK$ is $\Lambda$-cotorsion and that $E(\Kinf)[p^\infty]$ is finite.
Then $H^1(K_{\widetilde{\Omega}}/\Kinf, E[p^{\infty}])$ has no proper $\Lambda$-submodules of finite index.
\end{Lemma}

\begin{proof}
When $\Kinf = K^{\cyc}$, the result is proven by Greenberg (see \cite[Proposition~4.9]{Gre99}).
The proof crucially uses the surjectivity of the restriction map \eqref{surjectivity equation}.
The argument generalizes verbatim.
\end{proof}

We recall an equivalent definition of the Selmer group due to Greenberg.
The $p$-primary Selmer group can also be defined as the kernel of the following \emph{global-to-local} map
\[
H^1( K_{\widetilde{\Omega}}/\Kinf, E[p^\infty])\rightarrow \prod_{v\in \widetilde{\Omega}}\cH_v(E/\Kinf).
\]
Here for each finite prime $v\in \widetilde{\Omega}\setminus S_p$, the local term is defined as follows
\[
\cH_v(E/\Kinf) = \prod_{v^\prime|v} H^1\left( {\Kinf}_{,v^\prime}, E[p^\infty]\right),
\]
where $K_{\infty, v'}$ is the union of all completions of number fields contained in $K_{\infty}$.
The local condition at primes $v\in S_p$ is more subtle.
Since $E$ has good ordinary reduction at all primes $v\in S_p$, its $p$-adic Tate module $T$ fits into a short exact sequence of $\op{G}_{K_v}$-modules
\[0\rightarrow T^+\rightarrow T \rightarrow T^-\rightarrow 0\] such that $T^{+}$ and $T^-$ are free of rank $1$ over $\Z_p$, and $T^-$ is unramified.
Identifying $E[p^{\infty}]$ with $T\otimes \Q_p/\Z_p$, write $D=T^-\otimes \Q_p/\Z_p$.
For $v\in S_p$, define
\[
\cH_v(E/\Kinf) = \bigoplus_{v^\prime|v} H^1\left( {\Kinf}_{,v^\prime}, E[p^\infty]\right)/L_{v^\prime}
\]
with 
\[
L_{v^\prime} = \ker\left( H^1\left( {\Kinf}_{,v^\prime}, E[p^\infty]\right) \rightarrow H^1\left( I_{v^\prime}, D\right)\right).
\]
Here $I_{v^\prime}$ denotes the inertia group at $v^\prime$.
This definition of the Selmer group is more useful to work with when proving results about congruent Galois representations.
The following result shows that the above definition of the Selmer group (due to Greenberg) matches the usual Selmer group.
Thus, we do not distinguish between them in this paper.

\begin{Lemma}
Let $v\in \widetilde{\Omega}$, then 
\[\mathcal{H}_v(E/\Kinf)=J_v(E/\Kinf).\]
\end{Lemma}
\begin{proof}
Suppose that $v\notin S_p$.
Let $\ell\neq p$ be a rational prime and $L$ be \textit{any} algebraic extension of $\Q_\ell$.
Then, \cite[Proposition~2.1]{Gre99} asserts that $E(L)\otimes \Q_p/\Z_p=0$.
In particular, we have that $E(K_{\infty, v'})\otimes \Q_p/\Z_p=0$, for $v'|v$.
Thus, $\mathcal{H}_v(E/\Kinf)=J_v(E/\Kinf)$.
When $v\in S_p$, the result follows from \cite[Proposition~2.4]{Gre99}.
\end{proof}
We now  introduce the $\Omega_0$-imprimitive Selmer group.
It is a generalized Selmer group obtained by imposing conditions only at primes $v\in \widetilde{\Omega}\setminus \Omega_0$.

\begin{Defi}
The $\Omega_0$ \emph{imprimitive Selmer group} is defined as follows
\[
\Sel^{\Omega_0}(E/\Kinf) := \ker\left( H^1\left( K_{\widetilde{\Omega}}/\Kinf, E[p^\infty]\right)\xrightarrow{\theta_0} \bigoplus_{v\in \widetilde{\Omega}\setminus \Omega_0} \mathcal{H}_v(E/\Kinf)\right).
\]
\end{Defi}

\begin{Lemma}
\label{localmuzero}
Let $v\in \Omega_0$, then $\cH_v(E/\Kinf)$ is a cofinitely generated and cotorsion $\Lambda$-module with $\mu$-invariant equal to $0$.
Equivalently, it is a cofinitely generated $\Z_p$-module. 
\end{Lemma}

\begin{proof}
We refer the reader to the standard argument on \cite[pp.~37--38]{GV00}.
\end{proof}
We now define the \emph{reduced} classical and imprimitive Selmer groups.
These are denoted by $\Sel(E[p]/\Kinf)$ and $\Sel^{\Omega_0}(E[p]/\Kinf)$, respectively.
Define
\[
\cH_v(E[p]/\Kinf) :=
\begin{cases}
\prod_{v^\prime|v} H^1\left( {\Kinf}_{,v^\prime}, E[p]\right) & \ \textrm{ if } v\in \widetilde{\Omega}\setminus S_p\\
\bigoplus_{v'|v} H^1(K_{\infty, v'}, E[p])/\overline{L}_{v'} & \ \textrm{ if } v\in S_p
\end{cases}
\]
where  \[\overline{L}_{v'}=\ker\left( H^1\left( {\Kinf}_{,v^\prime}, E[p]\right) \rightarrow H^1\left( I_{v^\prime}, D[p]\right)\right).\]
\begin{Defi}
The \emph{residual imprimitive Selmer group} is defined as follows
\[\Sel^{\Omega_0}(E[p]/\Kinf):=\op{ker} \left(H^1\left( K_{\widetilde{\Omega}}/\Kinf, E[p]\right)\xrightarrow{\overline{\theta}_0} \bigoplus_{v\in \widetilde{\Omega}\setminus \Omega_0} \cH_v(E[p]/\Kinf)\right).\]
\end{Defi}

\begin{Lemma}
\label{NSW lemma}
Let $G$ and $M$ be finite abelian groups of $p$-power order such that $G$ acts on $M$. Suppose that $M^G=0$, then $M=0$.
\end{Lemma}
\begin{proof}
The result follows from \cite[Proposition~1.6.12]{NSW}.
\end{proof}

\begin{Lemma}\label{no BS lemma}
Choosing an index $i\in \{1,2\}$, set $E:=E_i$.
Given a prime $v\in \tilde{\Omega}\backslash \Omega_0$, the natural map \[\iota_v: \cH_v(E[p]/\Kinf)\rightarrow \cH_v(E[p^{\infty}]/\Kinf)[p],\] induced by inclusion of $E[p]$ into $E[p^\infty]$, is injective.
\end{Lemma}
\begin{proof}
Recall that $\tilde{\Omega}$ is the set of primes $\mathfrak{T}\cup S_p$, where $\mathfrak{T}$ is the set of $v$ of $K$ at which either $E_1$ or $E_2$ or both have bad reduction.
The definition of $\Omega(E_i)$ is given in Definition \ref{def of Omega0}.
Since $v\nmid \Omega(E_i)$, we have a number of cases to consider.

\begin{enumerate}
    \item First, consider the case when $v\in S_p$.
    We consider the commutative square with injective horizontal maps
\[
\begin{tikzcd}[column sep = small, row sep = large]
 & \cH_v(E[p]/\Kinf) \arrow{r} \arrow{d}{\iota_v} & \bigoplus_{v'|v} H^1\left( I_{v^\prime}, D[p]\right) \arrow{d}{j_v} \\
 & \cH_v(E[p^{\infty}]/\Kinf)[p]\arrow{r} & \bigoplus_{v'|v} H^1\left( I_{v^\prime}, D\right)[p].
\end{tikzcd}\]
Since $D$ is unramified for primes $v\in S_p$, it follows that $H^0(I_{v'}, D)=D$ is divisible.
Thus, 
\[\ker(j_v) = H^0(I_{v'}, D)/p=0.
\] Therefore, $j_v$ is injective, and so is $\iota_v$.
\item Next, suppose that $v\in \mathfrak{T}$ is a prime at which $E$ has good reduction.
Since $v\notin \Omega_0$, and hence $v\notin \Omega(E)$, by assumption, we have that $E(\F_v)[p^\infty]=0$.
Since $v$ is a prime in $\mathfrak{T}$, it follows that $v\nmid p$ and thus the kernel of the reduction map
\[
E(K_v)\rightarrow E(\F_v)
\]
is pro-$\ell$ for $\ell\neq p$.
Hence, $E(\F_v)[p^\infty]=0$ implies that $E(K_v)[p^\infty]=0$.
Since $K_\infty/K$ is a pro-$p$ extension and $E(K_v)[p^\infty]=0$, it follows from Lemma~\ref{NSW lemma} that $E(K_{\infty,v'})[p^\infty]=0$ as well for any prime $v'|v$ of $K_\infty$.
Therefore, it follows from the Kummer sequence that the map $\iota_v$ must be injective.
\item Finally, consider the case when $v\in \mathfrak{T}$ is a prime at which $E$ has bad reduction.
Setting $K':=K(\mu_p)$, we denote by $K'_{\infty}$ the composite of $K_\infty$ with $K'$.
Since we assume that $v\notin \Omega(E)$, there is a prime $w|v$ of $K'$ such that $E$ does not have split multiplicative reduction at $w$.
Since $E$ has bad reduction at $v$, and $v$ is unramified in $K'$, it follows that $E$ has bad reduction at $w$.
Hence, we deduce that $E$ has either non-split multiplicative reduction, or additive reduction at $w$.
\par Choose any prime $v'$ of $K_\infty$ that lies above $v$, and consider the map 
\[
\alpha_{v'}:H^1(K_{\infty, v'}, E[p])\rightarrow H^1(K_{\infty, v'}, E[p^\infty])[p]
\]
which is induced by inclusion of $E[p]$ into $E[p^\infty]$.
Let $w'$ be the unique prime of $K'_{\infty}$ that lies above both $w$ and $v'$. Consider the map 
\[
\beta_{w'}:H^1(K'_{\infty, w'}, E[p])\rightarrow H^1(K'_{\infty, w'}, E[p^\infty])[p],
\]
also induced by inclusion of $E[p]$ into $E[p^\infty]$.
Note that these maps fit into a commutative square 
\[
\begin{tikzcd}[column sep = small, row sep = large]
 & H^1(K_{\infty, v'}, E[p]) \arrow{r} \arrow{d}{\alpha_{v'}} & H^1(K'_{\infty, w'}, E[p]) \arrow{d}{\beta_{w'}} \\
 & H^1(K_{\infty, v'}, E[p^\infty])[p]\arrow{r} & H^1(K'_{\infty, w'}, E[p^\infty])[p].
\end{tikzcd}\]
To show that $\iota_v$ is injective, it suffices to prove that $\alpha_{v'}$ is injective for all $v'|v$.
In the above diagram, the horizontal maps are restriction maps.
Since $K'/K$ has degree coprime to $p$, it follows that $K'_{\infty,w'}/K_{\infty,v'}$ also has degree coprime to $p$.
Therefore, from the inflation-restriction sequence, the horizontal restriction maps in the above diagram maps are injective.
As a consequence, if we show that $\beta_{w'}$ is injective, then it shall follow that $\alpha_{v'}$ is also injective.
Note that the kernel of $\beta_{w'}$ is $H^0(K'_{\infty,w'}, E[p^\infty])/p$.
Thus, it suffices to show that $ E(K'_{\infty,w'})[p^\infty]$ is $p$-divisible.
Since $K'$ contains $\mu_p$ and $p\geq 5$, it follows from \cite[Proposition~5.1]{HM99} that
\[
E(K'_{\infty,w'})[p^\infty]=0,
\] and the result follows from this.
\end{enumerate}
\end{proof}

\begin{Proposition}
\label{Kim 2.10}
Let $E$ be either $E_1$ or $E_2$ and assume that $E(K)[p]=0$.
Then,
\[
\Sel^{\Omega_0}(E[p]/\Kinf) \simeq \Sel^{\Omega_0}(E/\Kinf)[p].
\]
\end{Proposition}
\begin{proof}
Recall that $\widetilde{\Omega}$ consists of a finite set of primes containing the primes above $p$ and the primes at which either $E_1$ or $E_2$ has bad reduction.
In particular, $\Omega_0$ is contained in $\widetilde{\Omega}$.
Consider the diagram relating the two Selmer groups
\[
\begin{tikzcd}[column sep = small, row sep = large]
0\arrow{r} & \Sel^{\Omega_0}(E[p]/\Kinf) \arrow{r} \arrow{d}{f} & H^1(K_{\widetilde{\Omega}}/\Kinf, E[p]) \arrow{r} \arrow{d}{g} & \op{im} \overline{\theta}_0 \arrow{r} \arrow{d}{h} & 0\\
0\arrow{r} & \Sel^{\Omega_0}(E/\Kinf) [p] \arrow{r} & H^1(K_{\widetilde{\Omega}}/\Kinf, E[p^{\infty}])[p] \arrow{r} &\left(\op{im} \theta_0\right)[p]\arrow{r} & 0,
\end{tikzcd}\]
where the vertical maps are induced by the Kummer sequence.
Note that
\[
0 = H^0(K,E[p])=H^0(\Kinf,E[p])^{\Gamma}.
\]
Since $\Kinf/K$ is a pro-$p$ extension, it follows from Lemma~\ref{NSW lemma} that $H^0(\Kinf, E[p^{\infty}])=0$.
Therefore $g$ (and hence $f$) is injective.
On the other hand, it is clear that $g$ is surjective.
By an application of the snake lemma, it suffices to show that $h$ is injective.
This follows from Lemma~\ref{no BS lemma}.
\end{proof}

\begin{Lemma}
The isomorphism $E_1[p]\simeq E_2[p]$ of Galois modules induces an isomorphism of residual Selmer groups 
\[
\Sel^{\Omega_0}(E_1[p]/\Kinf)\simeq \Sel^{\Omega_0}(E_2[p]/\Kinf).
\]
\end{Lemma}
\begin{proof}
Let $T_i$ be the $p$-adic Tate module of $E_i$ and identify $E_i[p^{\infty}]$ with $T_i\otimes \Q_p/\Z_p$.
Note that $E_i[p]$ fits into a short exact sequence 
\[
0\rightarrow C_i[p]\rightarrow E_i[p]\rightarrow D_i[p]\rightarrow 0,
\] where $C_i=T_i^+\otimes \Q_p/\Z_p$ and $D_i=T_i^-\otimes \Q_p/\Z_p$.
The action of $\op{G}_{\Q_p}$ on $C_i$ is via $\chi \gamma_i$, where $\gamma_i$ is an unramified character and $\chi$ is the $p$-adic cyclotomic character.
On the other hand, the action on $D_i$ is via the unramified character $\gamma_i^{-1}$.

Let $\Phi:E_1[p]\xrightarrow{\sim} E_2[p]$ be a choice of isomorphism of Galois modules.
It induces an isomorphism $H^1(K_{\widetilde{\Omega}}/\Kinf, E_1[p])\xrightarrow{\sim} H^1(K_{\widetilde{\Omega}}/\Kinf, E_2[p])$.
To prove the lemma, it suffices to show that for $v\in \widetilde{\Omega}\setminus \Omega_0$, 
the isomorphism $\Phi$ induces an isomorphism 
\[
\mathcal{H}_v(E_1[p]/\Kinf)\xrightarrow{\sim} \mathcal{H}_v(E_2[p]/\Kinf).\]
This is clear for $v\neq p$.
For $v=p$, this follows from the fact that $\Phi$ induces an isomorphism $D_1[p]\xrightarrow{\sim} D_2[p]$.
\end{proof} 

\begin{Cor}\label{resiso}
\label{p-torsion of S_0 fine selmer are iso}
With notation as above,
\[\Sel^{\Omega_0}(E_1/\Kinf)[p]\simeq \Sel^{\Omega_0}(E_2/\Kinf)[p].\]
\end{Cor}

\begin{Th}\label{resisothm}
Let $E_1$ and $E_2$ be $p$-congruent elliptic curves defined over $K$ such that
\begin{enumerate}
 \item $E_1(K)[p]=0$, or equivalently, $E_2(K)[p]=0$.
 \item Assumption~\ref{finitely decomposed assumption} is satisfied.
\end{enumerate}
Then,
\[
\begin{split}&\Sel(E_1/\Kinf) \text{ is }\Lambda\text{-cotorsion with }\mu=0\\
\Leftrightarrow &\Sel(E_2/\Kinf) \text{ is }\Lambda\text{-cotorsion with }\mu=0.
\end{split}
\]
\end{Th}
\begin{proof}
For $i=1,2$, we have a left exact sequence
\[0\rightarrow \Sel(E_i/\Kinf)\rightarrow \Sel^{\Omega_0}(E_i/\Kinf)\rightarrow \bigoplus_{v\in \Omega_0} \mathcal{H}_v(E_i/\Kinf).\]
By Lemma~\ref{localmuzero}, we know that for $v\in \Omega_0$, the cohomology group $\mathcal{H}_v(E_i/\Kinf)$ is a cotorsion $\Lambda$-module with $\mu=0$.
Therefore, 
\[
\begin{split}&\Sel(E_i/\Kinf) \text{ is }\Lambda\text{-cotorsion with }\mu=0\\
\Leftrightarrow &\Sel^{\Omega_0}(E_i/\Kinf) \text{ is }\Lambda\text{-cotorsion with }\mu=0.
\end{split}
\]
Using an argument identical to \cite[Proposition~3.10]{KR21}, we can show that
\[\begin{split}&\Sel^{\Omega_0}(E_i/\Kinf) \text{ is }\Lambda\text{-cotorsion with }\mu=0\\
\Leftrightarrow &\Sel^{\Omega_0}(E_i/\Kinf)[p] \text{ is finite.} \end{split}\]
By Corollary~\ref{resiso}, we have the isomorphism 
\[\Sel^{\Omega_0}(E_1/\Kinf)[p]\simeq \Sel^{\Omega_0}(E_2/\Kinf)[p].\]
Thus, the assertion follows.
\end{proof}

For the rest of this section we assume that the conditions of Theorem~\ref{resisothm} are satisfied for both $E_1$ and $E_2$.
Furthermore, we impose the following assumption.
\begin{assumption}\label{cotorsionass}
The $p$-primary Selmer group, $\Sel(E_1/\Kinf)$ (or equivalently, $\Sel(E_2/\Kinf)$) is a finitely generated $\Z_p$-module.
\end{assumption}
We now show that $\lambda$-invariants of $\Sel^{\Omega_0}(E_1/\Kinf)$ and $\Sel^{\Omega_0}(E_2/\Kinf)$ coincide.
As before, we write $E$ to denote either $E_1$ or $E_2$.
Recall that in Lemma~\ref{analogue of HM 2.3}, we showed that the map defining the (classical) Selmer group is surjective.
It follows that
\begin{equation}
\label{quotient}
\Sel^{\Omega_0}(E/\Kinf)/\SelpK \simeq \bigoplus_{v\in \Omega_0}\cH_v(E/\Kinf).
\end{equation}
For $v\in \Omega_0$, Assumption~\ref{finitely decomposed assumption} guarantees that there are finitely many primes $v'|v$ in $\Kinf$.
Hence, for $v\in \Omega_0$, we have that $\cH_v(E/\Kinf)$ is a direct sum over a finite set of primes. 

\begin{Defi}\label{random defi needed to be cited}Let $\sigma_E^{(v)}$ denote the $\Z_p$-corank of $\cH_v(E/\Kinf)$ for $v\in \Omega_0$.
\end{Defi}
Let $\lambda^{\Omega_0}(E/\Kinf)$ be the $\lambda$-invariant of $\Sel^{\Omega_0}(E/\Kinf)$.
Therefore,
\begin{equation}
\label{relating im primitive and classical lambda invariant}
\lambda^{\Omega_0}(E/\Kinf)=\op{corank}_{\Z_p} \left(\Sel^{\Omega_0}(E/\Kinf)\right) = \lambdapK + \sum_{v\in \Omega_0} \sigma_E^{(v)}.
\end{equation}
The first equality follows from the structure theory of $\Lambda$-modules, and the second equality follows from \eqref{quotient}.

\begin{Lemma}
\label{lemma: Hp is cofree}
Let $E$ be either $E_1$ or $E_2$.
Then, $\cH_p(E/\Kinf)$ is $\Lambda$-cofree.
\end{Lemma}

\begin{proof}
The reader is referred to \cite[p.~23]{GV00} for the argument.
\end{proof}
\begin{Proposition}\label{nofinitelambdasubs}
Let $E$ be either $E_1$ or $E_2$.
Suppose that
\begin{enumerate}
 \item Assumptions \ref{finitely decomposed assumption} and \ref{cotorsionass} hold,
 \item $E(K)[p]=0$.
\end{enumerate}Then, the Selmer group $\Sel^{\Omega_0} (E/\Kinf)$ contains no proper finite index $\Lambda$-submodules.
\end{Proposition}
\begin{proof} We adapt the proof of \cite[Proposition~4.14]{Gre99}, which is due to Greenberg.
Recall that $\widetilde{\Omega}$ is a finite set of primes containing $S_p$ and the primes at which $E_1$ or $E_2$ has bad reduction.
Consider the \emph{strict} and \emph{relaxed} $p$-primary Selmer groups defined as follows
\[
\Sel^{\op{rel}}(E/\Kinf):=\op{ker}\left\{H^1(K_{\widetilde{\Omega}}/\Kinf,E[p^{\infty}])\rightarrow \bigoplus_{v\in \widetilde{\Omega}\setminus (\Omega_0\cup S_p)} \mathcal{H}_v(\Kinf, E[p^{\infty}])\right\},
\]
\[
\Sel^{\op{str}}(E/\Kinf):=\op{ker}\left\{\Sel^{\op{rel}}(E/\Kinf)\rightarrow \bigoplus_{v\in \Omega_0\cup S_p} \prod_{v'|v} H^1(K_{\infty, v'}, E[p^{\infty}])\right\}.
\]
Lemma~\ref{analogue of HM 2.3} asserts that the restriction map
\[
H^1(K_{\widetilde{\Omega}}/\Kinf,E[p^{\infty}])\rightarrow \bigoplus_{v\in \widetilde{\Omega}} \mathcal{H}_v(\Kinf, E[p^{\infty}])
\]
is surjective.
Therefore, the restriction map
\[
\Sel^{\op{rel}}(E/\Kinf)\rightarrow \bigoplus_{v\in \Omega_0\cup S_p} \mathcal{H}_v(\Kinf, E[p^{\infty}])
\]
is also surjective.
Observe that the kernel of the restriction map 
\[
\Sel^{\op{rel}}(E/\Kinf)\rightarrow \bigoplus_{v\in S_p} \mathcal{H}_v(\Kinf, E[p^{\infty}])
\]
is $\Sel^{\Omega_0}(E/\Kinf)$.
Therefore,
\[
\Sel^{\op{rel}}(E/\Kinf)/\Sel^{\Omega_0}(E/\Kinf)\simeq \bigoplus_{v\in S_p} \mathcal{H}_v(\Kinf, E[p^{\infty}])
\]
is a cofree $\Lambda$-module by Lemma~\ref{lemma: Hp is cofree}.
By \cite[Lemma~2.6]{GV00}, to complete the proof of the proposition, it suffices to show that $\Sel^{\op{rel}} (E/\Kinf)$ has no proper finite index $\Lambda$-submodules
But, $\Sel^{\op{str}}(E/\Kinf)^{\vee}$ is a quotient of $\Sel^{\Omega_0} (E/\Kinf)^{\vee}$; so it must also be $\Lambda$-torsion.

Recall that $\chi$ denotes the $p$-adic cyclotomic character.
Let $\omega$ be the Teichm\"uller lift of the mod-$p$ cyclotomic character and $\kappa:=\omega^{-1}\chi$.
For $s\in \Z$, let $A_s$ denote the twisted Galois module $E[p^{\infty}]\otimes \kappa^s$.
Since $E(K)[p]=0$ by assumption, Corollary~\ref{finiteness corollary} asserts that $H^0(\Kinf, E[p^{\infty}])=0$.
Further, since $A_s$ and $E[p^{\infty}]$ are isomorphic as $\Gal(\closure{K}/\Kinf)$-modules, it follows that $H^0(\Kinf, A_s)=0$.
For any subfield $\mathcal{K}$ of $\Kinf$, and $v$ a prime of $K$ which does not divide $p$, set 
\[
\mathcal{H}_v(\mathcal{K}, A_s)=\prod_{\eta|v} H^1(\mathcal{K}_{\eta}, A_s)/(A_s(\mathcal{K}_{\eta})\otimes \Q_p/\Z_p),
\]
where $\eta$ ranges over the primes of $\mathcal{K}$ above $v$.
Define the products
\begin{align*}
 P^{\widetilde{\Omega}, \op{rel}}(\mathcal{K}, A_s)&:=\bigoplus_{v\in \widetilde{\Omega}\setminus (\Omega_0\cup S_p)} \mathcal{H}_v(\mathcal{K}, A_s),\\
 P^{\widetilde{\Omega}, \op{str}}(\mathcal{K}, A_s)&:=\bigoplus_{v\in \widetilde{\Omega}\setminus (\Omega_0\cup S_p)} \mathcal{H}_v(\mathcal{K}, A_s)\times \bigoplus_{v\in \Omega_0\cup S_p} \left\{\prod_{\eta|v} H^1(\mathcal{K}_{\eta}, A_s)\right\}.
\end{align*}
Let $S_{A_s}^{\op{rel}}(\mathcal{K})$ and $S_{A_s}^{\op{str}}(\mathcal{K})$ be the Selmer groups defined as follows
\begin{align*}
 S_{A_s}^{\op{rel}}(\mathcal{K})&:=\ker\left(H^1(K_{\widetilde{\Omega}}/\mathcal{K}, A_s)\rightarrow P^{\widetilde{\Omega}, \op{rel}}(\mathcal{K}, A_s)\right),\\
 S_{A_s}^{\op{str}}(\mathcal{K})&:=\ker\left(H^1(K_{\widetilde{\Omega}}/\mathcal{K}, A_s)\rightarrow P^{\widetilde{\Omega}, \op{str}}(\mathcal{K}, A_s)\right).
\end{align*}
Since $\Sel^{\op{str}}(E/\Kinf)$ is $\Lambda$-cotorsion, we have that $S_{A_s}^{\op{str}}(\Kinf)^{\Gamma}$ is finite for all but finitely many values of $s$.
Hence, $S_{A_s}^{\op{str}}(K)$ is finite for all but finitely many values of $s$.
As in the proof of \cite[Proposition~4.14]{Gre99}, write $M=A_s$ and $M^* =A_{-s}$.
Write $S_M(K)$ for the Selmer group $S_{A_s}^{\op{rel}}(K)$.
By the discussion in \cite[p.~100]{Gre99}, we have that $S_{M^*}(K)$ is the strict Selmer group $S_{A_{-s}}^{\op{str}}(K)$.
Let $s$ be an integer such that $S_{M^* }(K)$ is finite.
Since $S_{M^*}(K)$ is finite and $M^* (K)=0$, the map $H^1(K_{\widetilde{\Omega}}/K, M)\rightarrow P^{\widetilde{\Omega}, \op{rel}}(K, M)$ is surjective (see \cite[Proposition~4.13]{Gre99}).
Arguing as in \cite[Proposition~4.14]{Gre99}, we see that $\Sel^{\op{rel}} (E/\Kinf)$ has no proper finite index $\Lambda$-submodules.
This completes the proof.
\end{proof}

\begin{Proposition}
\label{prop: compare lambdaS0}
Let $E_1$ and $E_2$ be $p$-congruent elliptic curves over $K$ such that
\begin{enumerate}
 \item Assumptions \ref{finitely decomposed assumption} and \ref{cotorsionass} are satisfied,
 \item $E_1(K)[p]=0$ (or equivalently, $E_2(K)[p]=0$).
\end{enumerate}
Then the imprimitive $\lambda$-invariants coincide, i.e.,
\[\lambda^{\Omega_0}(E_1/\Kinf)=\lambda^{\Omega_0}(E_2/\Kinf).\]
Equivalently,
\[
\lambda(E_1/\Kinf) + \sum_{v\in \Omega_0} \sigma_{E_1}^{(v)} = \lambda(E_2/\Kinf) + \sum_{v\in \Omega_0} \sigma_{E_2}^{(v)}.
\]
\end{Proposition}
\begin{proof}
By Assumption~\ref{cotorsionass}, $\Sel^{\Omega_0}(E_i/\Kinf)[p]$ is finite for $i=1,2$, and by Corollary~\ref{p-torsion of S_0 fine selmer are iso}, 
\begin{equation}
\label{eqn: p-torsion of Sel_S0 are isomorphic}
\Sel^{\Omega_0}(E_1/\Kinf)[p]\simeq \Sel^{\Omega_0}(E_2/\Kinf)[p].
\end{equation}
Proposition~\ref{nofinitelambdasubs} assets that $\Sel^{\Omega_0}(E_i/\Kinf)$ has no proper $\Lambda$-submodules of finite index for $i=1,2$.
Hence they are cofree $\Z_p$-modules.
Therefore, 
\[\Sel^{\Omega_0}(E_i/\Kinf) \simeq (\Q_p/\Z_p)^{\lambda_i},\]
where $\lambda_{i}:=\lambda^{\Omega_0}(E_i/\Kinf)$.
As a result,
\[
\lambda^{\Omega_0}(E_i/\Kinf)=\op{dim}_{\F_p} \left(\Sel^{\Omega_0}(E_i/\Kinf)[p]\right).
\]
The proposition follows from \eqref{eqn: p-torsion of Sel_S0 are isomorphic}.
\end{proof}

\section{Large $\lambda$-Invariant}
\label{section: large lambda invariant}
In this section, we use results from the previous section to construct elliptic curves with \emph{large} $5$-adic (anti-cyclotomic) $\lambda$-invariants.
All computations in this section can be found at \href{https://www.anweshray.com/_files/ugd/14df41_46a0a8c679a84902a43e5bc115df7e20.pdf}{the following hyperlink}.
Throughout this section, we work in the following setting.
Let $K=\Q(\sqrt{-1})$, $p=5$, and $\Kan$ be the anticyclotomic $\Z_5$-extension of $K$.
Recall that $\Sigma$ is the set of primes in $K$ that are finitely decomposed in $\Kan$.
Moreover, a prime $v\nmid 5$ lies in $\Sigma$ if and only if it is a split prime in $K$.
If $v\in \Sigma\setminus \{5\}$ and $v'\in v(\Kan)$, then $\Kan_{v'}$ is an unramified $\Z_5$-extension of $v$.
Therefore, $\Kan_{v'}=K_v^{\op{cyc}}$.

\begin{Lemma}\label{last lemma}
Let $E_1$ and $E_2$ be $5$-congruent elliptic curves defined over $K$.
Let $v\nmid 5$ be a prime such that $v\in \Sigma$.
Then the following assertions hold
\begin{enumerate}
\item If $E_1$ has good reduction at $v$ and $E_2$ has bad reduction at $v$, then
\[
\corank_{\Z_p} \cH_{v}\left(\Kan, E_1[5^\infty]\right) - \corank_{\Z_p} \cH_{v}\left(\Kan, E_2[5^\infty]\right) \geq 0.
\]
\item Furthermore, if $\op{Frob}_{v}$ acts trivially on $E_1[5]$, then
\[
\corank_{\Z_p} \cH_{v}\left(\Kan, E_1[5^\infty]\right) - \corank_{\Z_p} \cH_{v}\left(\Kan, E_2[5^\infty]\right) >0.
\] Furthermore, $E_2$ has split multiplicative reduction at $v$.
\end{enumerate}
\end{Lemma}

\begin{proof}
For (1) and the first assertion of (2), we refer the reader to the proof of \cite[Lemma~3.1]{Kim09}.
For the second assertion of (2), it follows from \cite[Lemma~2.3]{A17} that $E_2$ cannot have additive reduction at $v$.
Moreover, by \cite[Lemma~2.11]{A17}, if $E_2$ has non-split multiplicative reduction at $v$, then,
\[
\corank_{\Z_5} \cH_{v}\left(\Kan, E_1[5^\infty]\right) = \corank_{\Z_5} \cH_{v}\left(\Kan, E_2[5^\infty]\right).
\]
Hence, the strict inequality forces $E$ to have split multiplicative reduction at $v$.
\end{proof}

Consider the elliptic curve $E: y^2 = x^3 - x$.
This is the elliptic curve \href{https://www.lmfdb.org/EllipticCurve/Q/32a2/}{32a2} (Cremona label).
This is a rank 0 elliptic curve with good ordinary reduction at $p=5$ and CM by $\mathcal{O}_K$.
Rubin has proved that for all primes $q$, the $q$-part of the Selmer group is trivial (cf. \cite[p.~188]{Kim09}).
Over $K$, the elliptic curve $E$ has Mordell-Weil rank equal to $0$ and $\Sha(E/K)[5^{\infty}]=0$.
Since $E(K)[5]=0$, it follows that $\Sel(E/K)=0$.

\begin{Lemma}
\label{trivial Selmer}
The 5-primary Selmer group $\op{Sel}(E/\Kan)$ is equal to $0$.
\end{Lemma}
\begin{proof}
To prove the lemma we show that the natural map 
\[\op{Sel}(E/K)\rightarrow \op{Sel}(E/\Kan)^{\Gamma}\] is surjective.
It suffices to show that the local map
\[\gamma_v: J_v(E/K)\rightarrow J_v(E/\Kan)\]
is injective at all primes.
For $v\nmid 5$, the kernel of $\gamma_v$ has order at most $c_v^{(5)}$, where $c_v$ is the Tamagawa number and $c_v^{(5)}$ its $5$-primary part, see \cite[Proposition~4.1]{Gre03}.
For the given elliptic curve $E$, one checks that $c_v^{(5)}=0$ for all primes $v$.
Next, consider $\ker \gamma_v$ when $v|5$.
If $\widetilde{E}(\F_v)[5]=0$, then the map $\gamma_v$ is injective (see \cite[Proposition~3.5]{GCEC}).
Since $5$ splits in $K$, so $\F_v=\F_5$.
One can check that $\widetilde{E}(\F_5)\simeq \Z/2\times \Z/4$, so $\gamma_v$ is injective.
Thus, $\Sel(E/\Kan)^{\Gamma}=0$.
The assertion follows from an application of Nakayama's Lemma.
\end{proof}

In \cite[Theorem~5.3]{RS95}, Rubin and Silverberg have shown that for every rational number $t$, there exist explicit polynomials $a(t)$, $b(t)$, and $c(t)$ such that
\begin{align*}
E_t : y^2 &= x^3 + a(t)x^2 + b(t)x + c(t)\\
&  =x^3 + \left(t^9 -t\right)x^2 + \Big( -3125t^{20} - 39583t^{18} + 11875t^{16} - 95000t^{14} + 61750t^{12} \\
&  + 41166t^{10} + 12350t^8 - 3800t^6 + 95t^4 -63t^{2} -1\Big)x + \Big(-521875t^{29} - 1355787t^{27} \\
& -7366875t^{25} + 9635000t^{23} - 8315875t^{21} -3678639t^{19} - 10560675t^{17} + 30400t^{15} \\
& + 2091615t^{13} + 134479t^{11} + 62583t^9 - 14200t^7 + 2327t^5 + 107t^3 + 7t\Big)
\end{align*}
is an elliptic curve defined over $\Q$, with good reduction at $p=5$, discriminant
\[
\Delta(E_t) = 4^3 \left(5t^4 - 2t^2 +1\right)^5\left( 25t^8 -100t^6 -210t^4 -20t^2 +1\right)^5,
\]
and the additional property that $E_t[5] \simeq E[5]$.

Let $f(t) = (5t^4 - 2t^2 +1)( 25t^8 -100t^6 -210t^4 -20t^2 +1)$.
We can check that $\gcd(f(t),f^\prime(t)) =1$.
Let $\cA$ be the set of prime numbers $\ell>5$ such that $f(t)$ and $f^\prime(t)$ are prime to each other modulo $\ell$.
Let $\cA^\prime$ be the subset of $\cA$ consisting of those primes which split completely in $K(E[5])$.
By the Chebotarev density theorem, the set $\cA^\prime$ is infinite with positive density.
Note that primes in $\cA^\prime$ must split completely in $K$. 
Hence, these primes are finitely decomposed in $\Kan/K$.

Given a positive integer $k$, choose $\ell_1, \ldots, \ell_k\in \cA^\prime$.
Then, for all $1\leq i \leq k$ and $v|\ell_i$ in $K$, the Frobenius, denote by $\op{Frob}_v$, acts trivially on $E[5]$.
By arguments similar to \cite[p.~189]{Kim09}, there exists $t$ for which $E_t$ satisfies \emph{both} the properties
\begin{enumerate}
 \item $E[5]\simeq E_t[5]$.
 \item $E_t$ has bad reduction at each $\ell_i$.
\end{enumerate}


\begin{Lemma}
\label{check for torsion and mu}
With notation as above, suppose that $\Omega(E_t)\subseteq \Sigma$.
Then, the 5-primary Selmer group $\Sel(E_t/K)$ is $\Lambda$-cotorsion with $\mu=0$.
\end{Lemma}
\begin{proof}
Recall that $E(K)[5]=0$.
The elliptic curve $E$ has additive reduction at the only prime of bad reduction, which is 2.
Computations on Sage show that the Kodaira type at $v|2$ is $I_2^*$.
It follows from \cite[Theorem~3(2)]{dokchitser2012remark} that $E$ has additive reduction of type $I_2^*$ at the primes $w|v$ of $K'$.
Therefore, $\Omega(E)$ is empty.
By assumption, $\Omega(E_t)\subseteq \Sigma$.
Thus, we can apply Theorem~\ref{resisothm}.
The claim follows from Lemma~\ref{trivial Selmer}.
\end{proof}

Using Proposition~\ref{prop: compare lambdaS0} and \eqref{relating im primitive and classical lambda invariant}, we can compare $\lambda$-invariants of $\Sel(E/\Kan)$ and $\Sel(E_t/\Kan)$.
We obtain that,
\[
\lambda(E_t/\Kan) = \lambda(E/\Kan) + \sum_{v\in \Omega_0} \left(\sigma_{E}^{(v)} - \sigma_{E_t}^{(v)}\right) \geq \lambda(E/\Kan) + \sum_{v\in \Omega(E_t)} \left(\sigma_{E}^{(v)} - \sigma_{E_t}^{(v)}\right) 
\]
where $\sigma_{E}^{(v)}$ (resp. $\sigma_{E_{t}}^{(v)}$) denotes the $\Z_p$-corank of $\cH_v(E/\Kinf)$ (resp. $\cH_v(E_t/\Kinf)$).
It follows from our discussion that for any given integer $k$, we can construct an elliptic curve $E_t$ with bad reduction at $\ell_1, \ldots, \ell_k$ such that

\[
\lambda(E_t/\Kan) \geq \lambda(E/\Kan) + \sum_{i=1}^k \sum_{v|\ell_i} \left(\sigma_{E}^{(v)} - \sigma_{E_t}^{(v)}\right) \geq \lambda(E/\Kan) + 2k \geq 2k.
\]
We remark that we have used the fact that the rational primes $\ell_i$ split in $K$.

\subsection*{Example}
We work out an explicit example when $p=5$.
The number field $\Q(E[5])$ is defined by the polynomial
{\small
\[
\begin{split}& x^{32} - 40 x^{31} + 12000 x^{29} - 10760 x^{28} - 266800 x^{27} + 11468000 x^{26} \\ 
&  - 1096600000 x^{25} - 5566913400 x^{24} + 142092440000 x^{23}  + 4369939320000 x^{22}\\ 
& - 378135400000 x^{21} - 257156918536000 x^{20} - 6566363103320000 x^{19} \\
& - 68617519071000000 x^{18} + 71545316608000000 x^{17} 
+
10219246387427710000 x^{16} \\ 
&  + 120765118794306400000 x^{15} + 1192541074750776000000 x^{14} \\
&+ 11957464808846300000000 x^{13} + 82847553195104300000000 x^{12} \\ 
&+ 1265538538816402500000000 x^{11} + 17074052628101627500000000 x^{10}\\ 
& + 172311923953883400000000000 x^{9} + 1455441332493456906250000000 x^{8} \\ 
&+ 9447284167114381875000000000 x^{7} + 49318124991568150000000000000 x^{6}  \\ 
&+ 129093097413215156250000000000 x^{5} - 2350246524757890625000000000 x^{4} \\ 
&- 682311431690296875000000000000 x^{3} - 475786411839253906250000000000 x^{2} \\ 
& + 246274993230703125000000000000 x + 1450426905571313476562500000000.\end{split}
\]}
Since the extension is rather large it becomes difficult to find primes which split in it via direct calculation.
We appeal to known results about the splitting of primes in fields of the form $\Q(E[N])$.
The conductor of $E$ is $2^5=32$ and hence all prime numbers $\ell \neq 2$ are unramified in $\Q(E[5])$.
If $\ell$ is a prime number, different from 2 and 5, which splits in $\Q(E[5])$, then indeed it is easy to see that the mod-$5$ characteristic polynomial of $\op{Frob}_{\ell}$ must equal $(X-1)^2=X^2-2X+1$.
This is equivalent to requiring that 
\begin{enumerate}
    \item\label{triviality conditions 1} $\ell\equiv 1\mod{5}$, and
    \item\label{triviality conditions 2} $a_{\ell}(E)\equiv 2\mod{5}$, where $a_{\ell}(E):=\ell+1-\#E(\F_{\ell})$.
\end{enumerate}
This is however not a sufficient condition for $\ell$ to split in $\Q(E[5])$.
Let $\Phi_5(X,Y)$ be the modular function (see \cite[Section~11]{Cox11} for a precise definition).
Define
\[\Phi_5^E(X):=\Phi_5(X, j(E))=\Phi_5(X, 1728).\]
A prime $\ell$ splits in $\Q(E[5])$ if in addition to the above two conditions, 
\begin{enumerate}[resume]
    \item $\Phi_5^E(X)$ splits in $\F_{\ell}[X]$ into distinct linear factors.
\end{enumerate}
We refer the reader to \cite[Proposition~5.6.3]{adelmann04} for a proof of the claim.
For the calculation of $\Phi_5^E(X)$, we refer the reader to \cite[Section~13]{Cox11}.
Implementing the algorithms in \cite{BLS12, bruinier16}, the modular polynomials may be calculated.
The calculation was performed by A.~Sutherland and has been obtained from his website (see \href{https://math.mit.edu/~drew/modpolys/jfiles/phi_j_5.txt}{this link}).

We thus have that in particular
{\small
\[
\begin{split}
& \Phi_5(X,Y)\\
   =& 141359947154721358697753474691071362751004672000 \\
&+ 53274330803424425450420160273356509151232000 X \\
& - 264073457076620596259715790247978782949376 XY \\
&+  6692500042627997708487149415015068467200 X^2 \\
&+  36554736583949629295706472332656640000 X^2Y \\
&+ 5110941777552418083110765199360000 X^2Y^2\\
&+  280244777828439527804321565297868800 X^3\\
&-192457934618928299655108231168000 X^3Y\\
&+26898488858380731577417728000 X^3 Y^2-441206965512914835246100 X^3 Y^3\\
&+1284733132841424456253440 X^4+ 128541798906828816384000X^4 Y\\
&+ 383083609779811215375 X^4 Y^2+  107878928185336800 X^4 Y^3\\
&+ 1665999364600 X^4Y^4 + 1963211489280 X^5
- 246683410950 X^5 Y\\ & + 2028551200 X^5Y^2
-  4550940 X^5Y^3+  3720 X^5Y^4-  X^5 Y^5+ X^6.
\end{split}\]}
Recall that primes in the set $\cA^\prime$ also split in $K=\Q(\sqrt{-1})$, i.e., they satisfy \begin{enumerate}[resume]
    \item $\ell \equiv 1 \mod 4$.
\end{enumerate}
Using SageMath \cite{sagemath}, we find primes ($\leq 10^5$) satisfying conditions $(1)-(4)$.
There are two primes, namely $\ell_1 = 63241$ and $\ell_2 = 63901$.
The factorization of the modular polynomial in the corresponding polynomial rings is as follows
{\small
\begin{align*}
\Phi_5^E(X) &= (x + 9130) (x + 26600) (x + 28822) (x + 31643) (x + 37410) (x + 60303) \textrm{ in } \mathbb{F}_{63241}(X),\\
\Phi_5^E(X) &= (x + 15646) (x + 16523) (x + 16743) (x + 31583) (x + 36229) (x + 58255) \textrm{ in } \mathbb{F}_{63901}(X).
\end{align*}
}
There exists $t= 1059545078$, such that
\begin{enumerate}[label=(\roman*)]
    \item $E[5]\simeq E_t[5]$.
    \item $E_t$ has bad reduction at the primes above $\ell_1$ and $\ell_2$.
\end{enumerate}
The Weierstrass equation of $E_t$ is given by
\begin{align*}
E_{1059545078}: y^2 &= x^3 + a(1059545078)x^2 + b(1059545078)x + c(1059545078)\\
&= x^3 + 168296446137189087194501767921833910783492001998159631069\\
&\ \ \ \ 7153199761832534651819530 x^2 - 99366224513122955264858859328518\\
&\ \ \ \  66745867676250607317503754405604484994604040262265914871891763\\
&\ \ \ \ 67949528696260899356431612520935180977428921937667494177305190\\
&\ \ \ \ 6053441051323218105644624301x -
279273806942198208191162128630\\
&\ \ \ \ 62552373448346890426050059883907340261159020802156537659423817\\
&\ \ \ \ 30770268443471585676947896166109727494303843884645456030625365\\
&\ \ \ \ 28979753424554225598656831724236790266467819045996204975383860\\
&\ \ \ \ 9933396090752989113892249177868807612498499698298014.
\end{align*}
For this value of $t$, we find that
\begin{align*}
f(t) = &13 \times 401 \times 63241 \times 63901 \times 21068381440942021 \times 23007701426021875081 \\
& \times 24504438741475825204304998173516406719475833143478257969366221.
\end{align*}
Each of the prime factors of $f(1059545078)$ are of the form $1 \pmod{4}$.
Recall that these primes factors are the primes of bad reduction of $E_{1059545078}$.
Therefore, we see that \emph{all} the primes of bad reduction of $E_{1059545078}$ split completely in $K$, and hence are finitely decomposed in $\Kan$.

\begin{Lemma}
With respect to notation above, let $p=5$, $E_1=E$ and $E_2=E_{1059545078}$.
The Assumptions \ref{finitely decomposed assumption} and \ref{cotorsionass} are satisfied in this example.
\end{Lemma}

\begin{proof}
In Lemma~\ref{check for torsion and mu} we have already checked that $\Omega(E_1)=\emptyset$.
Note that all primes $\ell\neq 2$ at which $E_2$ has bad reduction are of the form $\ell \equiv 1\mod{4}$.
Thus, they split in $K$ and the primes $v|\ell$ lie in $\Sigma$.
Thus, the only prime $v\in \mathfrak{T}\backslash \Sigma$ is $v=(1+i)$, the unique prime of $K$ that lies above $2$.
We show that $v=(1+i)$ is not contained in $\Omega_0$.
Computations on Sage show that $E_2$ has additive reduction at $v$ of Kodaira type $I_2^*$.
By \cite[Theorem~3(2)]{dokchitser2012remark}, $E_2$ has additive reduction of type $I_2^*$ at the primes $w|v$ of $K'$.
Hence, by Definition \ref{def of Omega0}, we have that $v$ is not contained in $\Omega_0$.
In particular, $\Omega_0\subset \Sigma$.
This proves that Assumption~\ref{finitely decomposed assumption} is satisfied.

Assumption~\ref{cotorsionass} follows from Lemma~\ref{check for torsion and mu} and the fact that $\Omega_0$ is contained in $\Sigma$.
\end{proof}
Each of the primes $\ell_1$ and $\ell_2$ split in $K$, and thus give rise to $4$ primes for which Lemma~\ref{last lemma} applies.
The elliptic curve $E_1$ has good reduction at these primes, and $E_2$ has bad reduction at them, and we obtain that
\[
\lambda(E_{1059545078}/\Kan) \geq \lambda(E/\Kan) + 4 \geq 4
\]
since we know that $\Sel(E/\Kan)=0$.
\begin{Remark}
It was possible to find primes up to $10^7$ satisfying properties $(1)-(4)$, and there are 24 such primes: $\ell_1=63241$, $\ell_2=63901$, $\ell_3=514561$, $\ell_4=1311341$, $\ell_5=2399081$, $\ell_6=2502301$, $\ell_7=2620301$, $\ell_8=2790461$, $\ell_9=3325121$, $\ell_{10}=3436501$, $\ell_{11}=4046401$, $\ell_{12}=4050281$, $\ell_{13}=4559101$, $\ell_{14}=4800421$, $\ell_{15}=5403361$, $\ell_{16}=5609321$, $\ell_{17}=6660221$, $\ell_{18}=7601861$, $\ell_{19}=7959521$, $\ell_{20}=8942501$, $\ell_{21}=8959921$, $\ell_{22}=9181901$, $\ell_{23}=9187081$, and $\ell_{24}=9437321$.
However, the next part of the computation (i.e., finding $t$ satisfying $(i)$ and $(ii)$) becomes difficult (for us) with $k\geq 3$.
This is due to the product $\prod_{i=1}^k \ell_i$ being large when $k\geq 3$.
\end{Remark}
\section*{Acknowledgments}
The first named author acknowledges the support of the PIMS Postdoctoral Fellowship.
We thank the referee for carefully reading our manuscript.

\bibliographystyle{alpha}
\bibliography{references}
\end{document}